\documentclass{article}

\usepackage{microtype}
\usepackage{graphicx}
\usepackage{booktabs} 

\usepackage{hyperref}

\usepackage[accepted]{icml2024}

\usepackage{amsmath}
\usepackage{amssymb}
\usepackage{mathtools}
\usepackage{amsthm}

\usepackage[capitalize,noabbrev]{cleveref}

\theoremstyle{plain}
\newtheorem{theorem}{Theorem}[section]
\newtheorem{proposition}[theorem]{Proposition}
\newtheorem{lemma}[theorem]{Lemma}

\theoremstyle{definition}
\newtheorem{definition}[theorem]{Definition}
\newtheorem{assumption}[theorem]{Assumption}
\theoremstyle{remark}
\newtheorem{remark}[theorem]{Remark}

\usepackage[textsize=tiny]{todonotes}

\usepackage{subcaption}

\usepackage{hyperref}       
\usepackage{url}            
\usepackage{booktabs}       
\usepackage{amsfonts}       
\usepackage{nicefrac}       
\usepackage{microtype}      
\usepackage{xcolor}         

\usepackage{amsfonts}       
\usepackage{amssymb}
\usepackage{amsmath}
\usepackage{amsthm}
\usepackage{mathtools}
\usepackage{multicol}
\usepackage{multirow}

\usepackage{color}

\usepackage{array}

\usepackage[shortlabels]{enumitem}
\usepackage{makecell}
\usepackage{pifont}

\usepackage[symbol]{footmisc} 

\def\M{\mathcal{M}}

\def\D{\mathrm{D}}

\def\Exp{{\mathrm{Exp}}}

\def\sR{{\mathbb R}}
\def\sS{{\mathbb S}}
\def\sE{{\mathbb E}}

\def\sM{{\mathbb M}}

\def\gD{{\mathcal{D}}}

\def\gP{{\mathcal{P}}}

\def\gH{{\mathcal{H}}}

\def\gX{{\mathcal{X}}}

\def\gS{{\mathcal{S}}}

\def\gT{{\mathcal{T}}}
\def\gL{{\mathcal{L}}}

\def\gO{{\mathcal{O}}}
\def\gI{{\mathcal{I}}}

\def\expm{{\mathrm{expm}}}

\def\grad{{\mathrm{grad}}}

\newcommand{\trace}{\mathrm{tr}}

\DeclareMathOperator*{\argmin}{arg\,min}

\def\diag{{\mathrm{diag}}}
\def\dist{{\mathrm{dist}}}

\def\Retr{{\mathrm{Retr}}}

\def\Exp{{\rm Exp}}

\def\St{{\rm St}}
\def\Skew{{\rm Skew}}
\def\skew{{\rm skew}}
\def\sym{{\rm sym}}

\def\Sym{{\rm Sym}}
\def\Sp{{\rm Sp}}
\def\Gr{{\rm Gr}}

\usepackage{mathtools}

\makeatletter
\newcommand{\thickhline}{%
    \noalign {\ifnum 0=`}\fi \hrule height 1pt
    \futurelet \reserved@a \@xhline
}
\newcolumntype{"}{@{\hskip\tabcolsep\vrule width 1pt\hskip\tabcolsep}}
\makeatother

\icmltitlerunning{Riemannian Coordinate Descent Algorithms on Matrix Manifolds}

\begin{document}

\twocolumn[

\icmltitle{Riemannian Coordinate Descent Algorithms on Matrix Manifolds}





\begin{icmlauthorlist}
\icmlauthor{Andi Han}{riken}
\icmlauthor{Pratik Jawanpuria}{ms}
\icmlauthor{Bamdev Mishra}{ms}
\end{icmlauthorlist}

\icmlaffiliation{riken}{Riken AIP, Japan}
\icmlaffiliation{ms}{Microsoft India}

\icmlcorrespondingauthor{Andi Han}{andi.han@riken.jp}

\icmlkeywords{Machine Learning, ICML}

\vskip 0.3in
]



\printAffiliationsAndNotice{}  

\begin{abstract}

Many machine learning applications are naturally formulated as optimization problems on Riemannian manifolds. The main idea behind Riemannian optimization is to maintain the feasibility of the  variables while moving along a descent direction on the manifold. This results in updating all the variables at every iteration. 
In this work, we provide a general framework for developing computationally efficient coordinate descent (CD) algorithms on matrix manifolds that allows updating only a few variables at every iteration while adhering to the manifold constraint. In particular, we propose CD algorithms for various manifolds such as Stiefel, Grassmann, (generalized) hyperbolic, symplectic, and symmetric positive (semi)definite. While the cost per iteration of the proposed CD algorithms is low, we further develop a more efficient variant via a first-order approximation of the objective function. We analyze their convergence and complexity, and empirically illustrate their efficacy in several applications. 



\end{abstract}

\section{Introduction}

In this work, we consider the optimization problem 
\begin{equation}
    \min_{X \in \sR^{n \times p}} f(X) \qquad {\rm s.t. } \quad X \in \M, \label{main_prob_eq}
\end{equation}
where $\M$ is a smooth, and often nonlinear constraint. Examples of $\M$ include 
orthogonality constraint \citep{edelman1998geometry}, 
positive (semi)definite constraint \citep{bhatia2009positive,han2021riemannian}, 
fixed-rank constraint \citep{vandereycken2013low}, 
hyperbolic constraint \citep{nickel2018learning}, 
doubly stochastic constraint \citep{douik2019manifold}, etc. Problem~\eqref{main_prob_eq} has been explored in applications such as 
PCA \citep{zhang2016riemannian,kasai19a}, 
low-rank matrix/tensor completion \citep{jawanpuria18a,nimishakavi18a,kressner2014low}, computer vision \citep{pennec06a}, natural language processing \citep{jawanpuria19a,jawanpuria20a}, 
optimal transport \citep{mishra2021manifold,han2022riemannian,shi2021coupling}, 
and deep learning \citep{arjovsky2016unitary,wang2020orthogonal}. Problem~\eqref{main_prob_eq} has also been studied in various settings such as stochastic optimization \citep{bonnabel2013stochastic,zhang2016riemannian,tripuraneni2018averaging,sato2019riemannian,kasai19a,han2021improved}, differential privacy \citep{reimherr2021differential,han24a,utpala2022improved}, federated learning \cite{li22a,huang24a},  decentralized learning \citep{mishra19a}, and saddle point and bilevel optimization \citep{han23b,han2023nonconvexnonconcave,zhang2023sion,han2024framework}. 

The smooth constraint set can be turned into a Riemannian manifold by endowing a properly chosen metric structure. The Riemannian optimization approach \citep{absil2008optimization,boumal2023introduction} then provides a principled approach to solve \eqref{main_prob_eq} intrinsically on the manifold space. The main idea is to iteratively update the variable along a descent direction without leaving the manifold. The descent direction is often computed using the Riemannian gradient, which is then followed by a retraction update to ensure feasibility of the manifold constraint.  
As the dimensionality of the constraint set increases, ensuring feasibility via retraction becomes a key computational bottleneck, e.g., the complexity of ensuring orthogonality and positive definiteness scales as $O(n^3)$ with the input dimension $n$. 
This has led to many recent works \citep{gao2019parallelizable,xiao2021solving,ablin2023infeasible} that develop infeasible methods for solving \eqref{main_prob_eq}. However, such methods are largely limited to the orthogonality constraint and cannot be easily adapted to other manifolds. 

In the Euclidean space, the coordinate descent (CD) method \citep{luo1992convergence,nesterov2012efficiency,wright2015coordinate} is a classic algorithm that successively solves a small-dimensional subproblem along a component of the vector variable while holding others fixed. Since each subproblem can be more easily solved than the original problem, this strategy leads to efficient variable update. 

On manifolds, designing CD updates is inherently difficult \citep{gutman2023coordinate}. A few works have proposed manifold specific CD updates, mainly for the orthogonal \citep{shalit2014coordinate,jianggivens,massart2022coordinate} and Stiefel \citep{gutman2023coordinate} manifolds. {Although \citet{gutman2023coordinate} discuss a general framework for developing CD methods on manifolds, concrete developments have been shown only for the Stiefel manifold.} Recently, for a class of optimization objectives, \citet{darmwal2023low} have proposed CD updates on the symmetric positive definite manifold with the affine-invariant metric. 


In this work, we provide a general approach for developing CD algorithms on matrix manifolds. We summarize our contributions below.
\begin{itemize}
    \item We introduce a framework for designing CD algorithms on manifolds. In particular, we find a basis spanning the tangent space such that a chosen retraction along the direction of such a basis admits an efficient computation. We discuss a simple expression for the coordinate derivative. Finally, we provide optimization ingredients for various matrix manifolds of interest. 


    \item A nonlinear objective $f$ in (\ref{main_prob_eq}) requires computation of gradient for every CD update. Using a first-order approximation of $f$, we develop a more efficient CD algorithm which requires gradient computations one in every fixed number of CD updates.
    We analyze the convergence and complexity of the two algorithms with randomized and cyclic selection of coordinates.

    \item We show the benefits of the proposed CD algorithms on the orthogonal Procrustes, PCA, orthogonal deep network distillation, nearest matrix, and learning hyperbolic embeddings problems. 
\end{itemize}



        



    


\section{Preliminaries}

\textbf{Riemannian manifolds and optimization.} For a Riemannian manifold $\M$, denote its tangent space at $X \in \M$ as $T_X\M$. A Riemannian metric is an inner product structure $g_X(\cdot, \cdot) = \langle \cdot, \cdot \rangle_X : T_X\M \times T_X\M \rightarrow \sR$ that varies smoothly with the base point $X$. In this work, we particularly focus on matrix manifolds, i.e., where $X$ can be represented in the ambient vector space $\sR^{m \times n}$. The orthogonal projection ${\rm Proj}_X: \sR^{m \times n} \rightarrow T_X\M$ projects arbitrary ambient vectors to the tangent space $T_X\M$ with respect to the Riemannian metric. For a differentiable function $f: \M \rightarrow \sR$, the Riemannian gradient at $X$ is defined as the tangent vector $\grad f(X) \in T_X\M$ such that $\langle U, \grad f(X)\rangle_X = \D f(X) [U], \forall U \in T_X\M$ where $\D f(X) [U] = \langle \nabla f(X), U\rangle$. 
A retraction $\Retr_X : T_X\M \xrightarrow{} \M$ allows points to move along the manifold, which satisfies the conditions: $\Retr_X(0) = X$ and $\D \Retr_X(0)[U] = U$. 



\textbf{Related works.} We provide a detailed review of the existing coordinate descent (CD) algorithms on specific manifolds, along with other related works in Appendix \ref{appendix:sec:related_works}.

\textbf{Notations.} We use $\langle \cdot, \cdot \rangle$ without the subscript to represent the Euclidean inner product while we use $\langle \cdot, \cdot \rangle_X$ to denote the Riemannian inner product on $T_X\M$. The specific expression for $\langle \cdot, \cdot \rangle_X$ depends on both $\M$ and $X$. $\Sym(n)$ and $\Skew(n)$ denote the sets of $n \times n$ symmetric and skew-symmetric matrices, respectively. Let $\sym(A) \coloneqq (A + A^\top)/2$, $\skew(A) \coloneqq (A - A^\top)/2$, $\exp(\cdot)$ be the elementwise exponential, and $\expm(\cdot)$ be the matrix exponential. We also use $e_i$ to represent the $i$-th basis vector with the dimension to be determined from the context. $[A]_{ij}$ denotes the $i,j$-th entry of a matrix $A$ while $A_{ij}$ represents a matrix with index $i,j$. We use $I_n$ to denote the $n \times n$ identity matrix, $1_n$ to denote the size-$n$ vector of all 1s, and define $[n] \coloneqq \{1, 2, ...,n \}$.

\section{Proposed CD Framework}
\label{main_sect_cd}



As shown in \citep{shalit2014coordinate, gutman2023coordinate, massart2022coordinate, jianggivens,darmwal2023low}, for specific manifolds, the key in developing CD algorithms is the choice of the basis vectors $B_\ell$ ($\ell \in \gI$ and $\gI$ denotes the index set) spanning the tangent space that allow efficient retraction. In general, our chosen basis need not be orthonormal with respect to the Riemannian metric.
%
%
%
%
%
Once the basis and retraction are chosen, the CD update is given by $\Retr_X(-\eta \theta B_\ell)$, where $\eta > 0$ is the stepsize and $\theta$ is the coordinate derivative, i.e.,
\begin{equation}\label{eq:main_coordinate_derivative}
    \theta \coloneqq \frac{d}{d\theta} f(\Retr_X(\theta B_\ell) ) \vert_{\theta = 0} = \langle \grad f(X), B_\ell \rangle_X.
\end{equation}
It can be verified that $-\theta B_\ell$ is indeed a descent direction, i.e., 
$\langle \grad f(X), - \theta B_\ell \rangle_X = -\theta \frac{d}{d\theta} f(\Retr_X(\theta B_\ell) ) \vert_{\theta = 0} = -\theta^2 \leq 0$. The CD algorithm then involves iteratively selecting coordinate index $\ell$, computing $\theta$, and updating in the coordinate descent direction $\Retr_X(- \eta \theta B_\ell)$.

The main challenges in developing CD algorithms on matrix manifolds are: 1) characterization of $B_\ell$ which facilitates efficient computation, 2) efficient computation of $\theta$, and 3) easy generalization to different manifolds. We propose to leverage the following connection between the Riemannian and Euclidean gradients:
\begin{equation}\label{eq:theta_definition}
    \theta = \langle \grad f(X), B_\ell \rangle_X =  \langle \nabla f(X), B_\ell \rangle,
\end{equation}
where $\nabla f(X)$ is the Euclidean gradient and the last equality follows from the definition of the Riemannian gradient. We exploit \eqref{eq:theta_definition} to efficiently compute $\theta$ for several manifolds as it is independent of the Riemannian gradient and metric. 
In the subsequent sections, we develop concrete CD optimization ingredients for the manifolds of interest under the proposed approach. These are summarized in Table \ref{table_summary_cdmfd}.

%

\begin{table*}[t]
\begin{center}
{\small 
\caption{Summary of CD ingredients over various manifolds. $H_{ij} = e_ie_j^\top - e_je_i^\top$ and $E_{ij} = e_i e_j + e_j e_i^\top$ are the basis for skew-symmetric and symmetric matrices, respectively. 
$G_{ij}(\theta)$ and $R_{ij}(\theta)$ corresponds to the Givens and hyperbolic rotations, respectively. ${\rm P}_X \coloneqq I_n - XX^\top$, and $\nabla f(X)_k$ is the $k$-th column of $\nabla f(X)$. We use $\Exp_{x}^\gS(v)$ to denote the exponential retraction over sphere.
The complexity only considers the computation of coordinate derivative and coordinate update, while excluding the complexity of first-order oracle $\nabla f(X)$.
}
\label{table_summary_cdmfd}
\begin{tabular}{@{}p{0.11\textwidth}p{0.07\textwidth}p{0.12\textwidth}p{0.22\textwidth}p{0.25\textwidth}p{0.1\textwidth}@{}}
\toprule
 & \makecell[l]{Size} & \makecell[l]{Basis \\ $B_\ell$} &\makecell[l]{Coordinate derivative \\ $\theta =\langle \grad f(X), B_\ell \rangle_X$} & \makecell[l]{Coordinate descent update \\ $\Retr_X(-\eta \theta B_\ell)$} & \makecell[l]{Complexity \\ (per update)} \\
\midrule
Orthogonal${}^*$ & $n \times n$ & $H_{ij}X$ & $\langle \nabla f(X), H_{ij}X \rangle$ & $G_{ij} \big(-\eta \theta \big) X$ & $O(n)$ \\
\addlinespace[4pt]
Stiefel${}^\dagger$ & $n \times p$ & \makecell[l]{$ X H_{ij}$ ,  \\ $v_{X} e_k^\top$} &\makecell[l]{$ \langle \nabla f(X) , X H_{ij}\rangle $ , \\ ${\rm P}_{X}( \nabla f(X)_k)$} & \makecell[l]{$X G_{ij}(-\eta \langle \nabla f(X) , X H_{ij}\rangle)$ , \\ $\Exp_{X_{k}}^{\gS}\big(- \eta {\rm P}_{X}( \nabla f(X)_k) \big)$} & $O(n),O(np)$\\
\addlinespace[4pt]
SPD${}^\diamond$  & $n \times n$ & $L {E}_{ij} L^\top$ & $\langle \nabla f(X), L {E}_{ij} L^\top \rangle$ & $L \expm(- \eta \theta {E}_{ij}) L^\top$ & $O(n)$ \\\hline
\addlinespace[4pt]
\makecell[l]{Stiefel \\ Grassmann} & $n\times p$ & $H_{ij}X$ & $\langle \nabla f(X), H_{ij}X \rangle$ & $G_{ij} \big(-\eta \theta \big) X$ & $O(p)$ \\
\addlinespace[1pt]
Hyperbolic & $n \times p$ & $H_{ij} JX$ & $\langle \nabla f(X), H_{ij} J X \rangle$ &  $\begin{cases}
    G_{ij}\big( - \eta \theta  \big) X &\text{ if } i \neq 1 \\
    R_{ij}\big( - \eta  \theta \big) X &\text{ if } i = 1
\end{cases}$ &  $O(p)$ \\
\addlinespace[1pt]
Symplectic & $ 2n \times 2p$ & $E_{ij} \Omega_n X$ & $ \langle \nabla f(X), E_{ij} \Omega_n X \rangle$  & Proposition \ref{prop_retr_sympl_cd} & $O(p)$
\\\addlinespace[1pt]
\makecell[l]{Doubly stoch.} & $m \times n$ & $A e_i e_j^\top B^\top$ & $\langle \nabla f(YY^\top), A e_i e_j^\top B^\top  \rangle$ & ${\rm cSK}(X \odot \exp( -\eta \theta B_\ell \oslash X))$ & $O(1)$
\\
\makecell[l]{Multinomial} & $n \times p$ & $e_i (e_j - e_{j+1})^\top$ & $[\nabla f(X)]_{ij} - [\nabla f(X)]_{i(j+1)}$ & $\gP(X \odot \exp( -\eta \theta B_\ell \oslash X))$ & $O(1)$
\\
\makecell[l]{SPSD / SPD} & $n \times n$ & $e_i e_j^\top$ & $\langle \nabla f(YY^\top), e_i e_j^\top  \rangle$ & $ Y - \eta \theta e_ie_j^\top$ & $O(1)$ \\
\bottomrule
\end{tabular}
}
\end{center}
{\footnotesize ${}^*$\citep{shalit2014coordinate}; ${}^\dagger$\citep{gutman2023coordinate}; ${}^\diamond$\citep{darmwal2023low}.}
\end{table*}


\subsection{CD on Stiefel manifold}\label{sec:stiefel}




The Stiefel manifold $\St(n,p)$ is the set of column orthonormal matrices of size $\sR^{n \times p}$, i.e., $\St(n,p) \coloneqq \{ X \in \sR^{n \times p} : X^\top X = I_p \}$. When $p = n$, $\St(n,n) \equiv \gO(n)$, the orthogonal manifold. The tangent space of Stiefel manifold is identified as $T_X\St(n,p) = \{ U \in \sR^{n \times p} : X^\top U + U^\top X = 0 \}$. The Riemannian metric is defined as $\langle U, V \rangle_X \coloneqq \langle U, V\rangle$ for any $U, V \in T_X\M$. The Riemannian gradient is derived as  $\grad f(X) = \nabla f(X) - X \sym(X^\top \nabla f(X))$. 

\textbf{Choice of basis.}
Taking inspiration from $\gO(n)$, we adopt the $\Omega X$ parameterization of the tangent vectors (where $\Omega \in \Skew(n)$) and choose the basis as $B_\ell = H_{ij} X$ for $\ell \in \gI = \{ (i,j) : 1 \leq i < j \leq n \}$ and $H_{ij} \coloneqq e_i e_j^\top - e_j e_i^\top$. In contrast to $\gO(n)$, the chosen basis is not orthonormal for $\St(n,p)$. This is expected as the manifold $\St(n,p)$ has a dimension $np - \frac{p(p-1)}{2}$ while we adopt an over-parameterization of the tangent space using $n(n-1)/2$ basis vectors. 

\textbf{Retraction.}
For the purpose of CD update, 
we first note that $\Retr_X(tU) = \expm(t \Omega) X$ is a valid retraction on $\St(n,p)$ because: 1) $\Retr_X(0) = X$ and 2) $\D\Retr_X(0)[U] = \Omega X = U$ are satisfied \citep{siegel2019accelerated}.

\textbf{CD update.}
Based on the above choices of the basis vectors and retraction, the proposed CD update is $\Retr_X(- \eta \theta B_\ell) = G_{ij}(-\eta \theta)X$, where $\theta  = \langle \nabla f(X), B_\ell \rangle = [\nabla f(X) X^\top - X \nabla f(X)^\top]_{ij}$. Here, $G_{ij}(\theta) = I_n + (\cos(\theta) - 1) (e_ie_i^\top + e_je_j^\top) + \sin(\theta) (e_ie_j^\top - e_j e_i^\top)$ is known as the Givens rotation around axes $i, j$ with angle $-\theta$. Overall, each CD update only requires $O(p)$ as we modify only two rows of $X$.




\begin{remark}
\citet{gutman2023coordinate} propose a column-wise CD update on the Stiefel manifold which costs $O(np)$ per iteration. 
On the other hand, our proposed CD update is row-wise and costs $O(p)$, which is cheaper. Furthermore, the CD update strategy of \citep{gutman2023coordinate} cannot be applied to the sphere manifold, i.e., when $p = 1$, it reduces to the full gradient update on the sphere. This, however, is not an issue for our update. Finally, the update of \citet{gutman2023coordinate} is not invariant to the right action of orthogonal group and hence does not yield a valid CD strategy for the Grassmann manifold. In contrast, as shown in the next section, our strategy can be readily generalized to the Grassmann manifold. 

\end{remark}





\subsection{CD on Grassmann manifold} \label{sec:grassmann}

The Grassmann manifold $\Gr(n,p)$ represents the $p$-dimensional subspaces in $\sR^n$, which can be represented by an $n \times p$ orthonormal matrix $X$, i.e., $X \in \St(n,p)$, where the columns span the subspace. The representation is not unique, with $XQ$ representing the same subspace for arbitrary $Q \in \gO(p)$. Thus, the Grassmann manifold can be identified as $\Gr(n,p) = \{ [X] : X \in \sR^{n \times p}, X^\top X = I_p \}$, where $[X] \coloneqq \{ XQ : Q \in \gO(p) \}$. The tangent space can be uniquely characterized by the horizontal space at $T_X\St(n,p)$, i.e., $T_{[X]}\Gr(n,p) = \{ [U] : X^\top U = 0 \}$. For a given $\xi\in T_{[X]}\Gr(n,p)$, its unique horizontal lift is $U = {\rm lift}_X(\xi)$, where $[X]$ is represented as $X$. The lift operator satisfies ${\rm lift}_{XQ} (\xi) = {\rm lift}_X(\xi) Q$. On  $\Gr(n,p)$, the Riemannian metric is pushed forward by the Euclidean metric on $\St(n,p)$ as $\langle \xi, \zeta \rangle_{[X]} = \langle {\rm lift}_X(\xi), {\rm lift}_X(\zeta) \rangle$ and the corresponding Riemannian gradient $\grad f([X])$ can be represented by ${\rm lift}_X(\grad f([X]) = (I_n - XX^\top) \nabla f(X)$. Retractions such as QR retraction for $\St(n,p)$ also work for $\Gr(n,p)$ as long as it preserves the equivalence class, i.e., $[\Retr_{XQ}(t \, {\rm lift}_{X}(\xi) Q)] = [\Retr_X(t \, {\rm lift}_{X}(\xi))]$ for any $Q \in \gO(p)$. 
%
%
%
Below, we show that the proposed CD update for $\St(n,p)$ is also well-defined for $\Gr(n,p)$.
\begin{proposition}
\label{grass_cd_prop}
Consider a function $f : \Gr(n,p) \rightarrow \sR$.
Let the coordinate descent update at $[X]$ be given by $\Retr_{X}(-\eta \, \theta H_{ij} X) \coloneqq G_{ij}(- \eta \theta) X$ for $1\leq i < j \leq n$, where $\theta = \langle \nabla f(X), H_{ij} X \rangle$ and for some fixed stepsize $\eta > 0$. Then, $\Retr_{X Q} (-\eta \, \theta_{XQ} H_{ij} XQ) = \Retr_X(- \eta \theta H_{ij} X) Q$.
\end{proposition}





\subsection{CD on hyperbolic manifold}
\label{hyperbolic_space_sect}


We now consider the generalized hyperbolic space \citep{Bai2014MinimizationPA,xiao2023dissolving} $\gH(n, p) \coloneqq \{ X \in \sR^{n\times p} : -X^\top J X = I_p\}$, where $J = \diag(-1,1,...,1) \in \sR^{n\times n}$ is the metric tensor. When $p = 1$, this reduces to the well-known hyperbolic space (the hyperboloid model). 
The tangent space at $X \in \gH(n,p)$ is identified as
\begin{equation*}
\begin{array}{lll}
    T_X \gH(n,p) &=& \{ U \in \sR^{n\times p} : U^\top J X + X^\top J U = 0 \} \\
    &=& \{ W JX: W \in \Skew(n) \}.
\end{array}
\end{equation*}
%
The Riemannian metric on $T\gH(n,p)$ is the generalized Lorentz inner product as $\langle U, V \rangle_\gL \coloneqq \trace(U^\top J V)$.
The normal space is given by $N_X\gH(d,r) = \{ X S : S \in \Sym(p) \}$. The orthogonal projection to $T_X \gH(n,p)$ and the Riemannian gradient are derived below.
\begin{proposition}
\label{prop_hyper_ortho_grad}
The orthogonal projection of $A \in \sR^{n\times p}$ to $T_X\gH(n,p)$ is given by ${\rm Proj}_X(A) = A + X\sym(X^\top J A)$. The Riemannian gradient is $\grad f(X) = J \nabla f(X) + X \sym(X^\top \nabla f(X))$.
\end{proposition}

\textbf{Choice of basis.}
For generalized hyperbolic manifold, we consider the basis $B_\ell = H_{ij} J X= (e_i e_j^\top - e_j e_i^\top)JX$, for $1\leq i < j \leq n$. 

\textbf{Retraction.}
Taking inspiration from our Stiefel analysis in Section \ref{sec:stiefel}, we define the map $\Retr_X(tU) \coloneqq \expm(t W J) X$ for $U = WJX \in T_X\gH(n,p) $. We next show such a map defines a valid retraction. As shown below, the retraction expression considerably simplifies along the chosen basis $H_{ij}JX$.


\begin{proposition}
\label{prop_expm_retr_hyper}
 Given a tangent vector $U = WJX \in T_X \gH(n,p)$ for some skew-symmetric matrix $W$, then $\Retr_X(t U) \coloneqq \expm(tWJ)X$ is a retraction.
\end{proposition}


In fact, $\expm(tWJ)$ is a Lorentz transform that satisfies $\expm(tWJ)^\top J \expm(tWJ) = J$, which preserves the Lorentz inner product as $(LX)^\top J LX = X^\top J X = -I_p$.  Hence by following the direction $U = \theta H_{ij} J X$, we define a coordinate type of updates on (generalized) hyperbolic manifold as $\expm(\theta H_{ij} J) X$, which can be computed efficiently similar to the Givens rotation. Particularly, when $i, j \neq 1$, $H_{ij} J = H_{ij}$ and $\expm(\theta H_{ij} J) =  G_{ij}(\theta)$ exactly recovers the Givens rotation. When $i = 1$, we have $H_{ij} J = E_{ij} \coloneqq e_i e_j^\top + e_j e_i^\top$.  We show in the following lemma that this also leads to a rotation known as the hyperbolic rotation.

\begin{lemma}
\label{lemma_eij_lorentz}
For $U = \theta H_{ij} J X$ with $1\leq i < j \leq n$, $\Retr_X(tU) = \begin{cases}
    G_{ij}(\theta) X &\text{if } i \neq 1, \\
    R_{ij}(\theta) X &\text{if } i = 1
\end{cases}$
where $R_{ij}(\theta) = I_n + (\cosh(\theta) -1) (e_i e_i^\top + e_je_j^\top) + \sinh(\theta) (e_i e_j^\top + e_j e_i^\top)$. 
\end{lemma}
When $d = 4$, $R_{ij}(\theta)$ is known as the Lorentz boost with rapidity $-\theta$ and can be thought of as rotation in the time domain. Hence, while the Givens rotation based CD updates have been explored for the orthogonal and Stiefel manifolds \citep{shalit2014coordinate,gutman2023coordinate}, our approach generalizes the Givens rotation based CD updates to hyperbolic spaces. 

\textbf{CD update.}
Similar to the Stiefel case, the proposed CD update is $\Retr_X (-\eta \theta H_{ij} JX)$, where $\theta = \langle \nabla f(X), H_{ij} JX \rangle = [\nabla f(X) X^\top J - J X \nabla f(X)^\top]_{ij}.$ In Appendix \ref{appendix:sec:hyperbolic}, we additionally derive a canonical-type metric and a Cayley retraction.

\subsection{CD on symplectic manifold}

The symplectic manifold \citep{gao2021geometry,gao2021riemannian,gao2022optimization} is defined as $\Sp(n, p) \coloneqq \{ X \in \sR^{2n \times 2p} : X^\top \Omega_n X = \Omega_p \},$ where $\Omega_k \coloneqq \begin{bmatrix}
    0 & I_k \\
    -I_k & 0
\end{bmatrix}$ is a $2k \times 2k$ skew-symmetric block matrix. The tangent space is given as 
\begin{equation*}
\begin{array}{l}
    T_X \Sp(n,p) = \{U \in \sR^{2n \times 2p} :  U^\top \Omega_n X + X^\top \Omega_n U = 0\} \\
    \ \ \ \qquad\qquad= \{ S \Omega_n X : S \in \Sym(2n)  \}.
\end{array}
\end{equation*}
Here we consider the Euclidean metric \citep{gao2021geometry} as $\langle U, V\rangle_X = \trace(U^\top V)$ for any $X \in \Sp(n,p), U, V \in T_X \Sp(n,p)$. The Riemannian gradient \citep[Proposition 3]{gao2021geometry} is given by $\grad f(X) = \nabla f(X) - \Omega_n X W_X$, where $W_X \in \Skew(2p)$ is the unique solution to the Lyapunov equation $X^\top X W + W X^\top X = 2\, \skew (X^\top \Omega_n^\top \nabla f(X))$.

\textbf{Choice of basis.}
Similar to the Stiefel and hyperbolic manifolds, we consider the basis $B_\ell = E_{ij} \Omega_n X$ for the tangent space $T_X \Sp(n,p)$, where $E_{ij} = e_ie_j^\top + e_j e_i^\top$, for $1\leq i \leq j \leq 2n$. Here, $e_i$ is the $i$-th basis in $\sR^{2n}$. 

\textbf{Retraction.}
We propose the following retraction for efficient CD updates. 

\begin{proposition}
\label{prop_exp_retr}
For any $X \in \Sp(n,p)$ and $U = S \Omega_n X \in T_X \Sp(n,p)$, the map $\Retr_X(tU) = \expm(tS \Omega_n) X$ is a retraction.
\end{proposition}
The above retraction further simplifies when moving along the chosen basis direction.
\begin{proposition}
\label{prop_retr_sympl_cd}
Let $U_{ij} = E_{ij} \Omega_n X \in T_X \Sp(n,p)$ for $1\leq i \leq j \leq 2n$. Then, we have $\Retr_X(\theta U_{ij}) = X + (\exp{(- \theta)} - 1) e_ie_i^\top X + (\exp(\theta) - 1) e_{n+i} e_{n+i}^\top X,$ if $i = j - n, j > n$ and $\Retr_X(\theta U_{ij}) = X + \theta E_{ij} \Omega_n X$ otherwise.
\end{proposition}

\begin{remark}[Block coordinate updates]
\label{rmk_block_symp}
We may also consider block coordinate updates. 
Let $E = \begin{bmatrix}
    A & C \\
    C^\top &B
\end{bmatrix}$ for $A, B \in \Sym(n)$ and $C \in \sR^{n \times n}$, and we wish to update $X$ in the direction of $U = E \Omega_n X \in T_X\Sp(n,p)$. First, we consider the upper-left and bottom-right blocks, i.e., where $E = \begin{bmatrix}
    A & 0\\
    0 & 0
\end{bmatrix}$ or $E = \begin{bmatrix}
    0 & 0\\
    0 & B
\end{bmatrix}$ for arbitrary $A, B \in \Sym(n)$. Here, $\Retr_X(\theta E \Omega_n X) = X + \theta E \Omega_n X$. 
Second, if $E = \begin{bmatrix}
    0 & \diag(u) \\
    \diag(v) &0
\end{bmatrix}$, for $u,v \in \sR^{n}$, then $E\Omega_n = \begin{bmatrix}
    -\diag(u) & 0 \\
    0 &\diag(v)
\end{bmatrix}$, and therefore, $\Retr_X(\theta E\Omega_n X) = (I + \theta E \Omega_n + \frac{\theta^2}{2!} (E\Omega_n)^2 + \cdots)X = G X$, where $G = \begin{bmatrix}
    \diag(\exp{(-\theta u)}) &0 \\
    0 &\diag(\exp{(\theta v)})
\end{bmatrix}$. 
\end{remark}

\textbf{CD update.} Finally, based on the above discussion our proposed CD update is $\Retr_X(- \eta \theta E_{ij}\Omega_n X)$, where,
$
    \theta  = \langle \nabla f(X), E_{ij} \Omega_n X \rangle 
    = [\nabla f(X) X^\top \Omega_n^\top + \Omega_n X \nabla f(X)^\top]_{i,j}.
$

\subsection{CD on doubly stochastic and multinomial manifolds}

Given two marginals $\mu \in \Delta_m, \nu \in \Delta_n$ where $\Delta_k \coloneqq \{ z \in \sR^k : z \geq 0, z^\top 1_k = 1 \}$ denotes the $k$-simplex, the doubly stochastic manifold \citep{douik2019manifold,shi2021coupling,mishra2021manifold} is defined as $\Pi(\mu, \nu) \coloneqq \{ X \in \sR^{m \times n} : X > 0,  X 1_n = \mu, X^\top 1_m = \nu \}$. The tangent space is $T_X \Pi(\mu, \nu) = \{ U \in \sR^{m \times n} : U 1_n = 0, U^\top 1_m = 0 \}$, which can be endowed with the Fisher metric as $\langle U,V \rangle_X = \trace(U^\top (V \oslash X) )$, where $\odot$ and $\oslash$ represent the elementwise product and division operations, respectively. The orthogonal projection is ${\rm Proj}_X(A) = A - (\alpha 1_n^\top + 1_m \beta^\top) \odot X$, where $\alpha \in \sR^m, \beta \in \sR^n$ are solutions to the linear system:
$ \alpha \odot \mu + X \beta = A1_n, \beta \odot \nu + X^\top \alpha = A^\top 1_m$.
The Riemannian gradient is given by $\grad f(X) = {\rm Proj}_X(X \odot \nabla f(X)) = X \odot \big( \nabla f(X) - (\alpha 1_n^\top + 1_m \beta^\top)  \big)$. 


\textbf{Choice of basis.}
We consider the parameterization of the tangent space as $T_X\Pi(\mu, \nu) = \{ A C B^\top : A \in \sR^{m \times (m-1)}, B \in \sR^{n \times (n-1)}, A^\top 1_m = 0, B^\top 1_n = 0, C \in \sR^{(m-1)\times (n-1)} \}$. We notice that the tangent space has a dimension $(m-1)\times(n-1)$, and hence, we can let $A = [e_1 - e_2, ..., e_{m-1} - e_{m}] \in \sR^{m \times (m-1)}, B = [e_1 - e_2, ..., e_{n-1} - e_n] \in \sR^{n \times (n-1)}$, where we denote $e_i$ as the $i$-th canonical basis for the corresponding vector space. Hence the tangent space is parameterized by $C \in \sR^{(m-1)\times(n-1)}$. The basis we consider $B_\ell = A e_i e_j^\top B^\top = (e_i - e_{i+1})(e_j - e_{j+1})^\top$ for $i \in [m-1],  j \in [n-1]$.

\textbf{Retraction.} We consider the Sinkhorn retraction applied in the direction of the basis as $\Retr_X(- \eta \theta B_\ell) = {\rm SK}(X \odot \exp(- \eta \theta B_\ell \oslash X))$. Here, the Sinkhorn algorithm ${\rm SK}(U)$ iteratively normalize rows and columns of $U$ according to the given marginals \citep{knight2008sinkhorn}. We notice the input to the Sinkhorn algorithm only modifies a $2\times 2$ sub-matrix of $X$. It, thus, suffices to apply the Sinkhorn algorithm to the $2\times 2$ sub-matrix with the modified marginals, which largely simplifies the computation compared to running the Sinkhorn algorithm for the entire input. To this end, we define the coordinate Sinkhorn, denoted as ${\rm SK}^{ij}(U)$ or simply ${\rm cSK}(U)$ if the coordinates are clear from context, as performing the Sinkhorn algorithm for the $2\times 2$ sub-matrix formed by indices $i,i+1$ and $j,j+1$ with marginals $([\mu]_i - \sum_{k \neq j, j+1} [U]_{i k}, [\mu]_{i+1} - \sum_{k \neq j, j+1}[U]_{(i+1)k})$ and $([\nu]_j - \sum_{l \neq i,i+1} [U]_{tj}, [\nu]_{j+1} - \sum_{l \neq i,i+1} [U]_{l(j+1)} )$. The other entries of $U$ remains unchanged. We show in the next proposition that applying coordinate Sinkhorn to the basis results in a valid retraction.

\begin{proposition}
\label{prop_sinkhorn_cd}
The coordinate Sinkhorn applied to the basis $B_\ell$, i.e., ${\rm cSK}(X \odot \exp( t B_\ell \oslash X))$ is a valid retraction in the direction of $B_\ell$.
\end{proposition}

We can further simplify the computation of ${\rm cSK}(X \odot \exp( t B_\ell \oslash X))$, which is equivalent to performing Sinkhorn on a $2\times 2$ matrix. Furthermore, in this case, we show in Lemma \ref{lemma_sink_2x2} that the Sinkhorn admits a closed-form solution.


\textbf{CD update.}
The CD update follows as ${\rm cSK}(X \odot \exp(- \eta \theta B_\ell \oslash X))$, where the coordinate derivative $\theta$ is computed as
$
    \theta  = \langle \nabla f(X) , B_\ell\rangle
    = [\nabla f(X)]_{ij} - [\nabla f(X)]_{i(j+1)} - [\nabla f(X)]_{(i+1)j} + [\nabla f(X)]_{(i+1)(j+1)}.
$

\begin{remark}[CD on multinomial manifold]
The developments in this section readily applies to the multinomial manifold \citep{douik2019manifold}, i.e., $\sM^{n,p} \coloneqq \{ X \in \sR^{n \times p} : X > 0, X 1_p = v \}$ where we assume $v = 1_n$ without loss of generality. The multinomial constraint corresponds to $n$ independent simplex constraints restricted to positive entries. The tangent space is $T_X\sM^{n,p} = \{ U \in \sR^{n \times p} : U 1_p = 0\} = \{ V B^\top : V \in \sR^{n \times (p-1)}, B \in \sR^{p \times (p-1)}, B^\top 1_p = 0  \}$. Thus, the basis is similarly given by $B_\ell = e_i (e_j - e_{j+1})^\top$. The Riemannian metric is the same Fisher metric. A retraction is given in $\Retr_X(tU) = \gP(X \odot(\exp(t U \oslash X)))$, where $\gP(V) = V \oslash (V 1_p 1_p^\top)$ denotes the row normalization. It should be noted that $\gP(V)$ is a special case of the Sinkhorn algorithm without column normalization. Thus, in the basis direction, we can define the coordinate projection by modifying only two entries per row. The coordinate derivative can be computed as $\theta = \langle \nabla f(X), B_\ell \rangle = [\nabla f(X)]_{ij} - [\nabla f(X)]_{i(j+1)}$. 
\end{remark}

\subsection{CD on positive (semi)definite manifold}\label{spsd_section}

The set of fixed-rank symmetric positive semi-definite manifold (SPSD) matrices \citep{vandereycken2009embedded,vandereycken2013riemannian,massart2020quotient} is defined as $\sS_{+}^{n,p} \coloneqq \{ X \in \sR^{n \times n} : X \succeq 0, {\rm rank}(X) = p \}$. When $p = n$, we recover the set of symmetric positive definite (SPD) matrices as $\sS_{+}^{n,n} \equiv \sS_{++}^n$. For the purpose of developing efficient CD updates on SPSD, we follow the parameterization purposed in \citep{massart2020quotient}, i.e., $X \in \sS_{+}^{n \times p}$ is factorized as $X = YY^\top$, $Y \in \sR^{n \times p}_*$, which is unique up to the right-action of the orthogonal group $\gO(p)$. The Riemannian gradient can be computed as $\grad f(Y)  = \nabla f(YY^\top) = 2\, \sym(\nabla f(YY^\top)) Y$ because the $T_Y \sR^{n \times p}_* \simeq \sR^{n \times p}$. 
The main advantage of this parameterization is its simple expression of retraction, i.e., $\Retr_{Y} (t \xi) = Y + t \xi$ \citep{massart2020quotient}.

\textbf{Choice of basis, retraction, and CD update.}
Using $X = YY^\top$, the optimization problem is on $\sR^{n \times p}_*$ with a simple retraction.
%
%
For the objective $f : \sS_+^{n,p} \rightarrow \sR$, we initialize $Y \in \sR^{n \times p}_*$ and update as ${\Retr}_Y ( - \eta \nabla f(YY^\top) ) = Y - \eta \nabla f(YY^\top)$ for some stepsize $\eta$.
We choose the basis to be $e_i e_j^\top$ for $i \in [n], j \in [p]$, which is orthonormal for the tangent space $T_Y\sR^{n\times p}_*$. The CD update is given by ${\Retr}_Y(- \eta \theta e_ie_j^\top)$, where $\theta = \langle \nabla f(YY^\top), e_i e_j^\top \rangle = [\nabla f(YY^\top)]_{ij}$. 
%
%
%
The simplicity of the geometry allows CD to be developed efficiently on $\sS_{+}^{n,p}$, which coincides with the Euclidean CD update in the Euclidean space. When $p = n$, we obtain a  CD update for the SPD manifold.

\textbf{CD update with the BW metric.}
The Bures-Wasserstein (BW) metric for the SPD manifold ($p  = n$) has been recently studied for various machine learning applications \citep{bhatia2019bures,han2021riemannian}. 
For the BW metric, the gradient descent update is $\Exp_{X}(- \eta \grad f(X)) = X - 2\eta (\nabla f(X) X + X \nabla f(X)) + 4 \eta^2 \nabla f(X) X \nabla f(X).$ 
Consider a basis $E_{ij}X + X E_{ij}$ where $E_{ij} = e_ie_j^\top + e_j e_i^\top$. The coordinate derivative is computed as $ \theta_{ij} = \langle E_{ij} X + X E_{ij}, \nabla f(X)  \rangle$. Finally, the CD update is given by $X - 2 \eta \theta_{ij} (E_{ij} X + X E_{ij}) + 4 \eta^2 \theta_{ij}^2 E_{ij} X E_{ij}$. Each CD update modifies two rows and two columns of $X$. 

\begin{remark}
For the SPD manifold ($p=n$), \citet{darmwal2023low} propose CD updates based on the affine-invariant metric and Cholesky factorization. They specifically focus on a class of  objective functions and show that the exponential map computations are efficient. In contrast, our choices of parameterization/metric directly leads to a faster retraction.
\end{remark}
 


\section{Algorithms and Analysis}
\label{algo_sect}


\begin{algorithm}[t]
 \caption{Riemannian coordinate descent (RCD/RCDlin)}
 \label{RCD_algo_all}
 \begin{algorithmic}[1]
  \STATE Initialize $X_0 \in \M$. $\gI$ denotes the index set. Given hyper-parameters $K,S$, and $\eta$. 
  \FOR{$k = 0,...,K-1$}
  \STATE $X_k^0 = X_k$.
  \FOR{$s = 0, ..., S-1$}
  \STATE Pick $\ell_k^s \in \gI$ and construct basis $B_{\ell_k^s}$ at $X_k^s$.
  \STATE $\theta_k^s  =  \left\{\begin{array}{l}
      \langle \nabla f(X_k^s), B_{\ell_k^s} \rangle, \text{ if } \text{RCD}, \\
      \langle \nabla f(X_k), B_{\ell_k^s} \rangle, \text{ if } \text{RCDlin.}
  \end{array}
  \right.
  $ 
  \STATE Update $X_{k}^{s+1} = \Retr_{X_k^s}(- \eta \theta_k^s B_{\ell_k^s})$. 
  \ENDFOR
  \STATE $X_{k+1} = X_k^S$.
  \ENDFOR
 \end{algorithmic} 
\end{algorithm}

\textbf{RCD.}
We present the proposed Riemannian coordinate descent (RCD) algorithm in Algorithm~\ref{RCD_algo_all}. 
%
%
The complexity of RCD per iteration is the complexity of one first-order oracle and the update complexity in Table \ref{table_summary_cdmfd}. Although for some problem settings, we may explore the structure of the objective to efficiently compute $\theta = \langle \nabla f(X), B_\ell \rangle$, for general problem instances, $\nabla f(X)$ becomes the main computational bottleneck.


\textbf{RCDLin.}
%
To reduce gradient computations in the RCD setup (especially for non-linear objectives), we also propose the Riemannian linearized coordinate descent (RCDlin) method in Algorithm~\ref{RCD_algo_all}. The main difference with RCD is that the variables are updated using an anchored gradient at $X_k$ (which does not change for inner iterations). This scheme is equivalent to taking a linearization of the original cost function at $X_k$, and in the inner iterations, we solve: $\min_{X \in \M} \big\{ g(X) \coloneqq f(X_k) + \langle \nabla f(X_k), X - X_k  \rangle \}$,
where the inner product and subtraction are defined in the ambient Euclidean space. Subsequently, the Euclidean gradient at $X_k^s$ is $\nabla g(X_k^s) = \nabla f(X_k)$, and thus, $\theta_k^s = \langle \nabla f(X_k), B_{\ell_k^s} \rangle$. For the randomized setting with $S = 1$, RCDlin is equivalent to RCD. Additionally, for linear problems where $\nabla f(X_k) = C$ ($C$ is some constant matrix), RCDlin also reduces to RCD.


\subsection*{Convergence and complexity of Algorithm~\ref{RCD_algo_all}}
We next discuss the convergence analysis of RCD and RCDlin. It follows the standard analysis for CD algorithms \citep{wright2015coordinate}. 
{Note that \citet{gutman2023coordinate} mostly consider the analysis of CD algorithms under exponential map and parallel transport operations. In contrast, we consider the more general retraction and vector transport operations.}
We also adapt our analysis for RCDlin. For brevity, our analysis is informally discussed here. The analysis is in a compact neighbourhood around a critical point, which is required for validating certain regularity assumptions, boundedness of basis and projection onto the basis, and smoothness of the objective (details in Appendix \ref{appendix:sec:proofs}).

\textbf{On RCD.}
We start by showing the convergence of RCD under randomized selection of basis and certain regularity assumptions. $L$ is the smoothness constant of the objective.

\begin{theorem}[Randomized RCD]
\label{thm_rcd_random}
Under mild assumptions, consider RCD with $S = 1$ and $\ell_k^s$ selected uniformly at random from $\gI$. Then, choosing $\eta = \Theta\big(\frac{1}{L} \big)$ leads to $\min_{ 0\leq k \leq K-1} \sE \| \grad f(X_k) \|^2_{X_k} \leq O \big(  \frac{ |\gI| L}{K} \big)$.
\end{theorem}


The convergence of RCD with cyclic selection of basis requires further assumptions that bound the difference of the constructed bases between tangent spaces. These are  reasonable given the compactness of the domain. 
\begin{theorem}[Cyclic RCD]
\label{thm_cyclic_rcd}
Under mild assumptions in addition to the ones required by randomized RCD, consider RCD with $S = |\gI|$ and $\ell_k^s = s+1$ for $s = 0, ..., |\gI| - 1$. Then, for $\eta = \Theta\big( \frac{1}{L} \big)$, we have $\min_{0 \leq  k \leq K-1} \| \grad f(X_k) \|_{X_k} \leq O \big( \frac{|\gI|^2 L}{K} \big)$.
\end{theorem}

\textbf{On RCDlin.}
The key idea here is to relate the coordinate derivative $\theta_k^s = \langle \nabla f(X_k), B_{\ell_k^s} \rangle$ to the correct descent derivative $\langle \nabla f(X_k^s), B_{\ell_k^s} \rangle$. 
In randomized settings, we can show the same convergence rate as RCD up to some additional constants regulated by the difference between $\theta_k^s$ and the descent direction. For the cyclic settings, however, we require $S = |\gI|$ in order to cycle through all the basis. 


\begin{theorem}[Randomized and cyclic RCDlin]
\label{thm_rand_cyc_rcdlin}
Under assumptions required in Theorems \ref{thm_rcd_random} $\&$ \ref{thm_cyclic_rcd}, 
suppose $\theta_k^s$ and $\langle \nabla f(X_k^s),  B_{\ell_k^s}\rangle$ are positively related.
Then, consider randomized RCDlin with $\ell_k^s$ selected uniformly at random from $\gI$. Choosing $\eta = \Theta(\frac{1}{L})$ leads to $\min_{0 \leq k \leq K-1, 0\leq s \leq S-1} \sE \| \grad f(X_k^s) \|_{X_k^s}^2 \leq \big(\frac{|\gI| L}{KS}\big)$. In addition, consider cyclic RCDlin with $S = |\gI|$ and $\ell_k^s = s+ 1$ for $s = 0,..., |\gI| - 1$. Also, if $\eta = \Theta\big( \frac{1}{L} \big)$, then $\min_{0\leq k \leq K-1} \| \grad f(X_k^s) \|^2_{X_k^s} \leq O\big(\frac{ |\gI|^2 L}{K}\big)$.
\end{theorem}

\textbf{Complexity analysis.}
Let the cost of computing the coordinate derivative $\theta$ and CD update be $\delta$ (last column of Table \ref{table_summary_cdmfd}). Then, the total computational cost of RCD and RCDlin is $O(KS(\mathcal{F} + \delta))$ and $O(K(\mathcal{F} + S \delta) )$, respectively, where $\mathcal{F}$ denotes the cost of computing $\nabla f(X)$. We note that the proposed algorithms can parallely update in disjoint basis directions. For example, in the Stiefel/Grassmann case, we can select $n/2$ non-overlapping index pairs, which results in $n/2$ independent Givens rotation, and can be parallelized.




\begin{figure}[!t]
    \centering
    \begin{subfigure}[b]{0.22\textwidth}
        \begin{subfigure}[t]{\textwidth}
            \includegraphics[width=\textwidth]{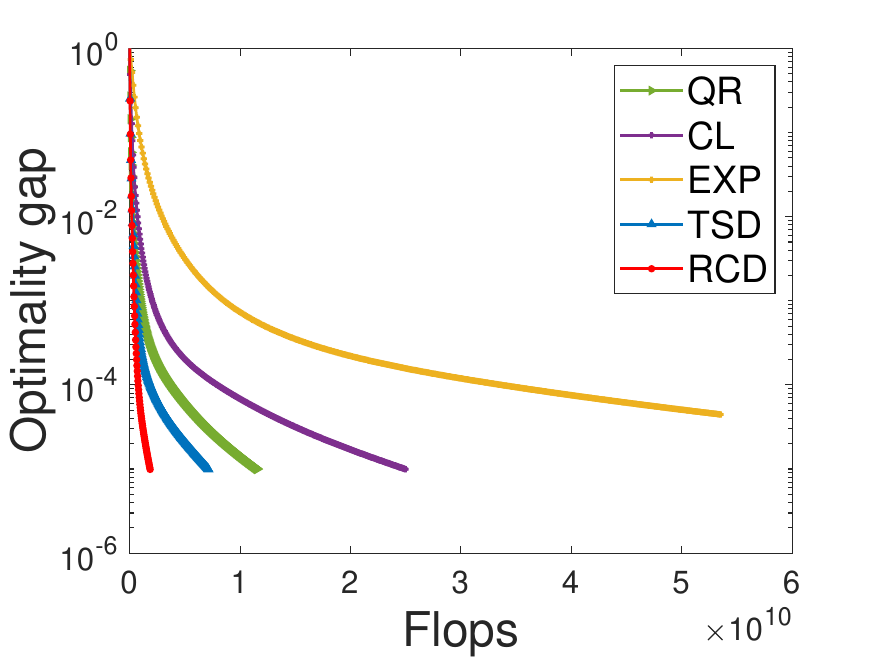}
        \end{subfigure}\vspace{.3ex}

        \begin{subfigure}[t]{\textwidth}
            \includegraphics[width=\textwidth]{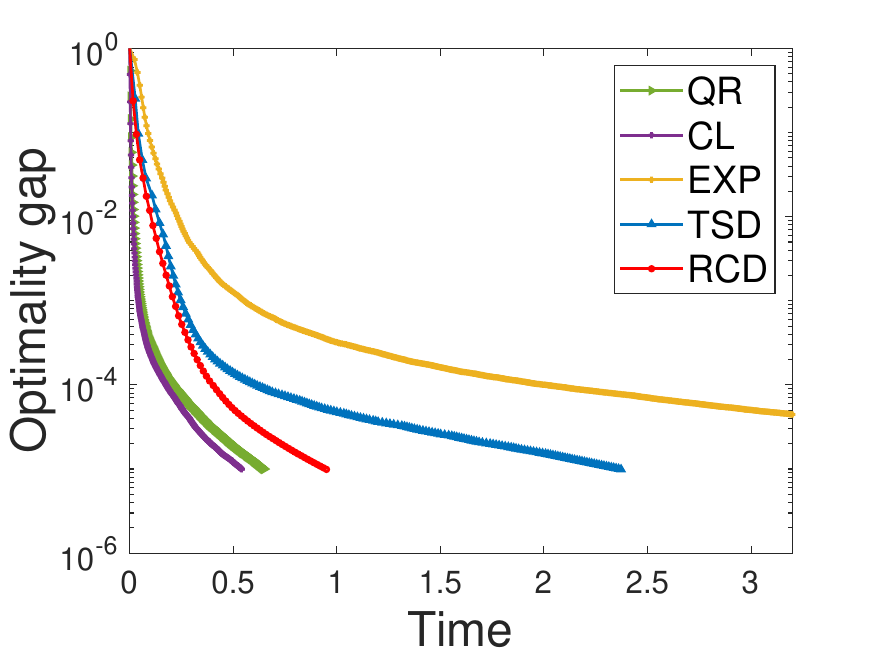}
        \end{subfigure}
        \caption{$n=200, p=150$}
    \end{subfigure}
    \begin{subfigure}[b]{0.22\textwidth}
        \begin{subfigure}[t]{\textwidth}
            \includegraphics[width=\textwidth]{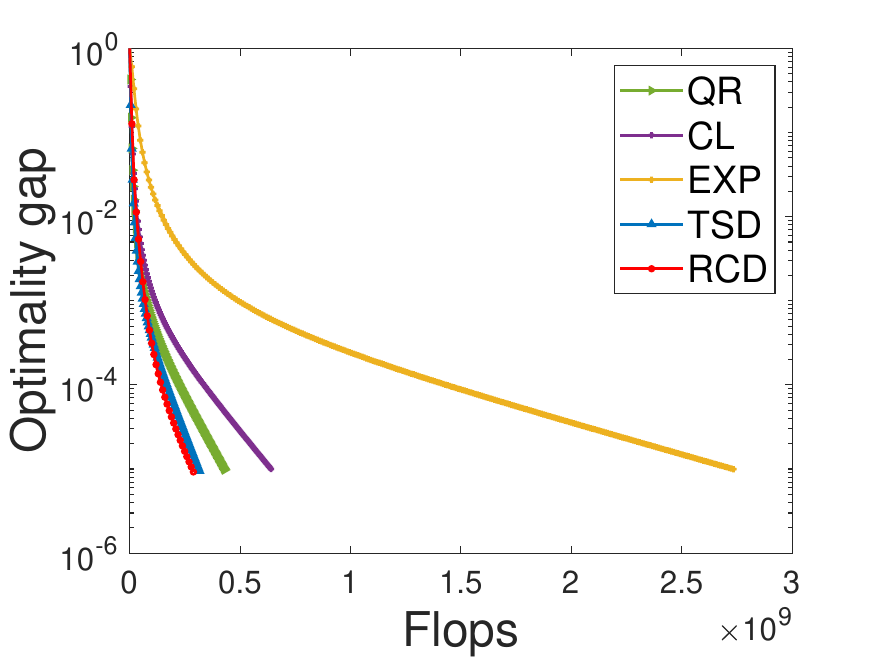}
        \end{subfigure}\vspace{.3ex}

        \begin{subfigure}[t]{\textwidth}
            \includegraphics[width=\textwidth]{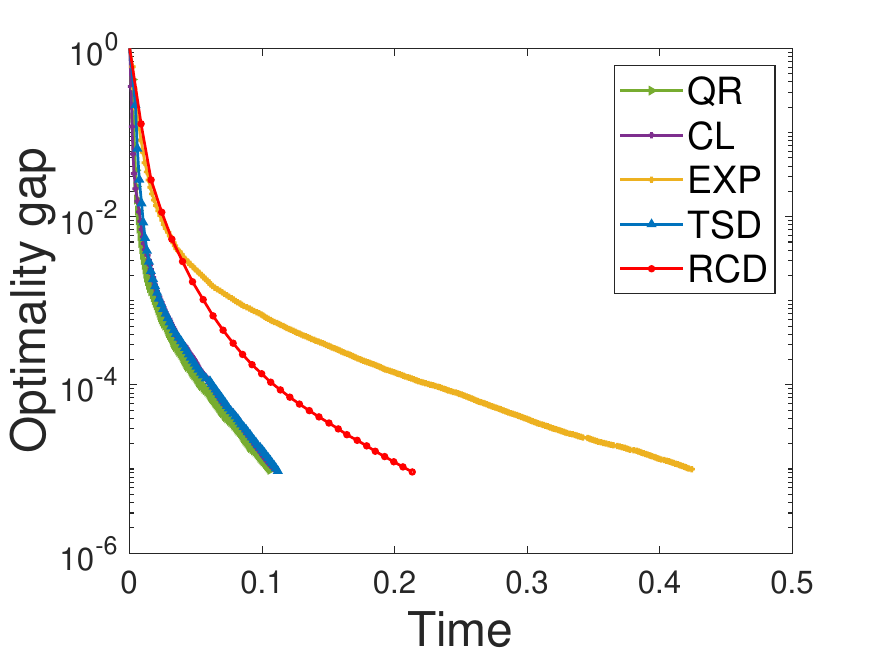}
        \end{subfigure}
        \caption{$n=200, p=50$}
    \end{subfigure}
    \caption{The Procrustes problem with varying $p$: (a) $p=150$ and (b) $p=50$. (Top row) Comparing various algorithms in terms of flop counts. (Bottom row) Comparing various algorithms in terms of runtime. We observe that our RCD algorithm obtains better flop counts than the baselines in flop counts and is competitive in terms of runtime.}
    \label{ortho_procru_fig}
\end{figure}

\begin{figure*}
    \centering
    \subfloat[PCA \label{pca_flops_fig}]{\includegraphics[width=0.22\textwidth]{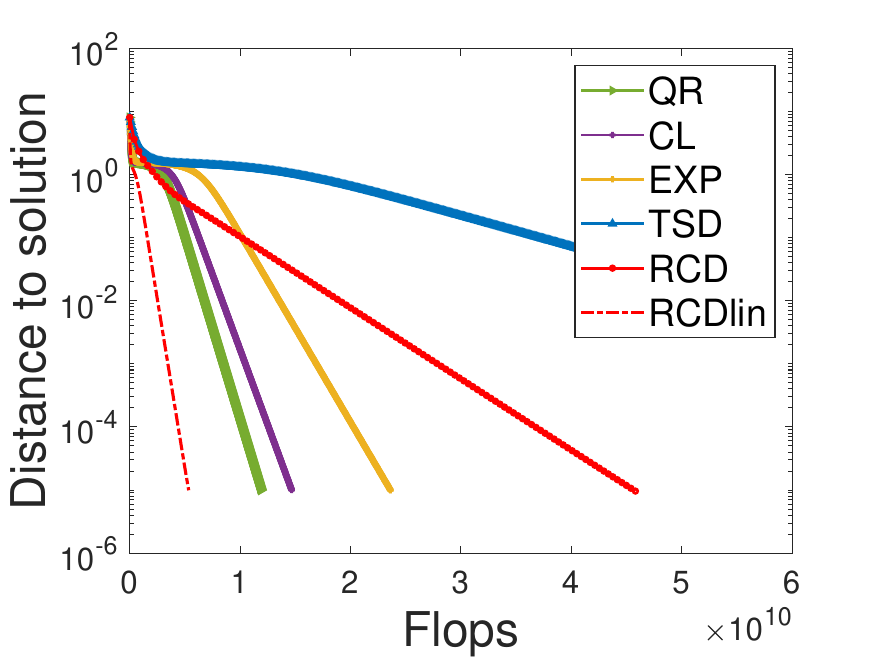}  }
    \subfloat[PCA \label{pca_basis_fig}]{\includegraphics[width=0.22\textwidth]{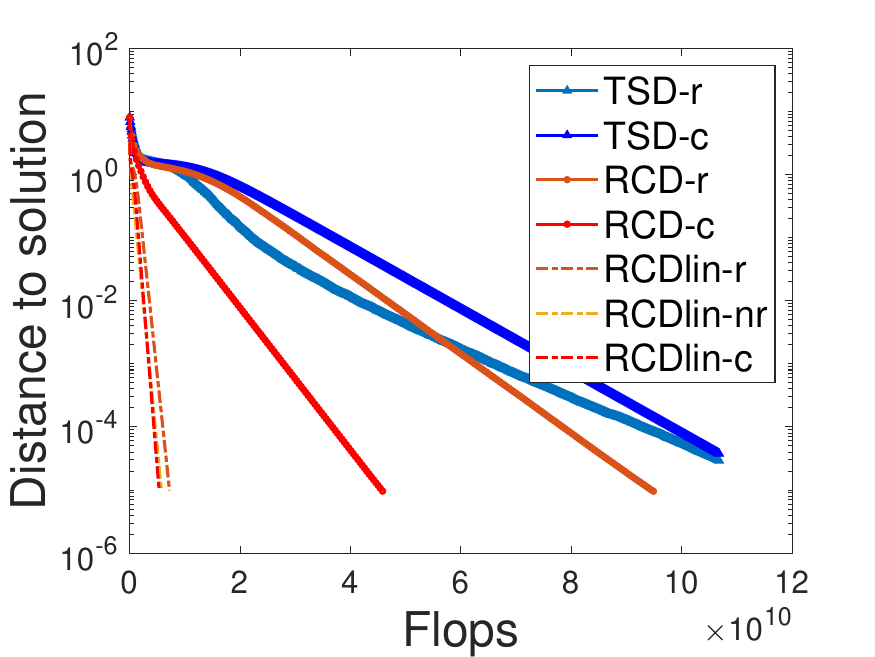}  }
    \subfloat[Procrustes \label{pro_basis_fig}]{\includegraphics[width=0.22\textwidth]{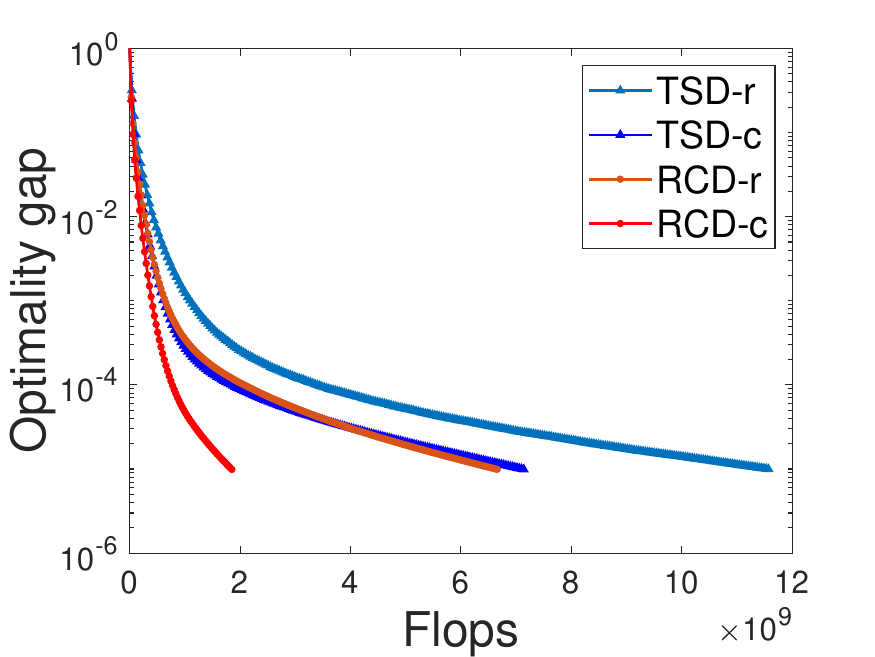}  }
    \subfloat[Procrustes \label{pro_infea_fig}]{\includegraphics[width=0.22\textwidth]{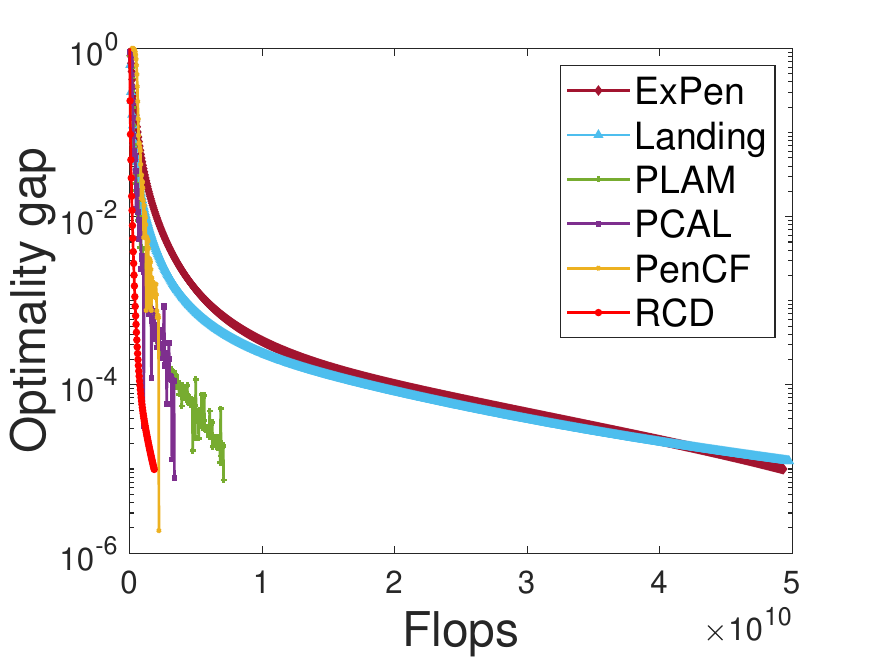} }
    \caption{(a) \& (b): Experiments on the PCA problem with $n = 200, p = 50$. In (a), we observe that our algorithm RCDlin achieves the fastest convergence due to low per-iteration cost. In (b), we compare various strategies for basis selection: cyclic selection (-c) and uniformly random selection (-r) of basis for TSD, RCD, and RCDlin, and selection without replacement (-nr) for RCDlin. We observe that cyclic and selection without replacement strategies are better than random selection. (c) \& (d): Experiments on the Procrustes problem with $n = 200, p = 150$. In (c), we again observe that cyclic selection performs better than random selection. In (d), RCD performs competitively against the infeasible methods.} 
    \label{grass_pca_fig}
\end{figure*}

\section{Experiments}
We now benchmark the performance of the proposed RCD and RCDlin algorithms in terms of computational efficiency (flop counts and/or runtime) and convergence quality (distance to optimality). 
One of the considered baselines is the Riemannian gradient descent method (RGD), a full gradient method. As RGD exploits the entire gradient direction, it has advantage over CD algorithms. However, RGD is significantly more costly than CD in every update. Our codes are implemented using the Manopt toolbox \citep{boumal2014manopt} and run on a laptop with an i5-10500 3.1GHz CPU processor. The codes are available at \url{https://github.com/andyjm3}.

\subsection{Orthogonal Procrustes and PCA}
\label{orthogonal_problem}

\textbf{Orthogonal Procrustes problem.}
We aim to solve $\min_{X \in \St(n,p)} \| X A - B \|^2 (\equiv -\langle XA, B\rangle )$ for given matrices $A \in \sR^{p \times p}$, $B \in \sR^{n \times p}$. There exists a closed-form solution provided by the (thin) SVD of $BA^\top$. For this, RCD and RCDlin have same updates as $\nabla f(X) = - BA^\top$ for $X \in \St(n,p)$. In experiments, we generate random matrices $A, B$ and evaluate the performance against optimality gap computed as $|f(X_k) - f^*|/|f^*|$.

\textbf{Baselines.} The closest baseline to RCD is TSD  \citep{gutman2023coordinate}, which is a CD method under an alternative construction of bases. As discussed, while TSD updates the columns, the proposed RCD updates the rows. Since RCD is equivalent to TSD for $p=n$, we focus only on the $p < n$ setting. For both RCD and TSD, we use the cyclic selection of basis. 
We also compare against RGD methods with QR, Cayley (CL), and exponential (EXP) retractions. 
For all the methods, we tune the stepsize. 

\textbf{Results.} In Figure \ref{ortho_procru_fig}, we show results with varying dimension $p$. 
While the proposed RCD obtains better flop counts than the baselines in flop counts, it is competitive in terms of runtime. 
We highlight that the runtime of RCD can be further improved with parallel implementation. 
In Figure \ref{pro_basis_fig}, we compare a variety of basis selection rules for both TSD and RCD: cyclic selection (`c') and uniformly random selection (`r'). 
We observe that cyclic selection is more favourable than the random selection rule for both the methods. 
We compare against full gradient infeasible methods in Figure~\ref{pro_infea_fig}, including PLAM \citep{gao2019parallelizable}, PCAL \citep{gao2019parallelizable}, PenCF \citep{xiao2022class}, ExPen \citep{xiao2021solving,xiao2023dissolving}, and Landing \citep{ablin2022fast,ablin2023infeasible}. RCD is performs competitively against infeasible methods for orthogonality constraints.





\textbf{PCA problem.}
The PCA problem solves a quadratic maximization problem as $\max_{X^\top X = I_p} \trace(X^\top A X)$ for some positive definite matrix $A$, i.e., $A \in \sS_{++}^n$. This problem is in fact an optimization problem over the Grassmann manifold because the objective is invariant to basis change of $X$. Hence, we use the Riemannian distance to the optimal solution on the Grassmann manifold to measure the performance. 
As discussed in Section~\ref{sec:grassmann}, our proposed RCD has well-defined updates on the Grassmann manifold. In contrast, TSD is not invariant to the basis change. 
For experiments, we generate $A$ with a condition number $10^3$ and with exponential decay of eigenvalues. For TSD, RCD, and RCDlin, we implement the cyclic selection of basis. 

\textbf{Results.} In Figure \ref{pca_flops_fig}, we observe that RCDlin achieves the best performance due to its low per-iteration cost. 
We note that TSD converges slowly due to non-invariance of the CD updates. In Figure~\ref{pca_basis_fig}, we compare the cyclic and uniformly random selection of the basis of RCD, RCDlin, and TSD. For RCDlin, we also implement the selection without replacement (`nr') strategy. 
We observe that cyclic and `nr' strategies are better than random selection. 



\begin{figure}[h]
\centering
\subfloat[Distill (flops)]{\includegraphics[width=0.22\textwidth]{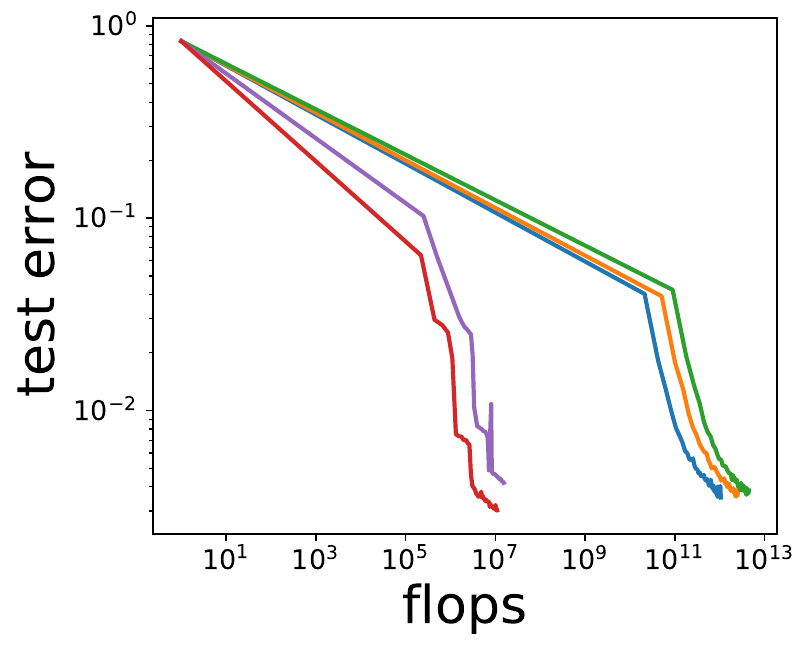}  }
    \subfloat[Distill (time) ]{\includegraphics[width=0.22\textwidth]{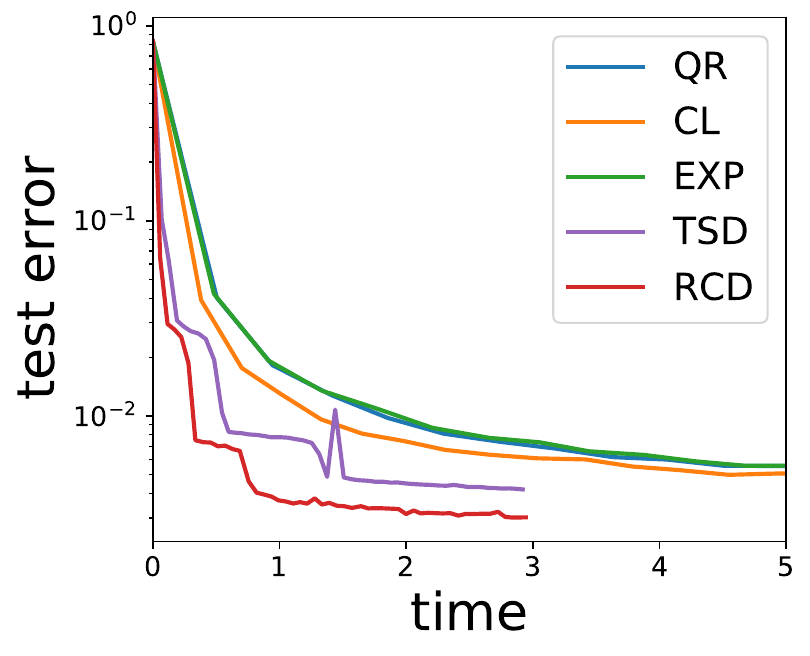}}
\caption{Experiments on the distillation problem. We observe that the proposed RCD algorithm performs better than the baselines both in terms of flop counts and runtime. }
\label{st_distill_fig}
\end{figure}

\subsection{Orthogonal deep networks for distillation}
We next evaluate RCD on a deep learning based distillation problem \citep{hinton2015distilling}. Let $\Theta$ denote the parameters of the student network (S) while $\Theta_{\rm T}$ be the optimal parameters of the teacher network (T). Then, the aim is to learn S that approximates T, i.e., minimize $\gL(\Theta) = \| \Psi_{\Theta}(X) - \Psi_{\Theta_{\rm T}}(X) \|^2$, where $\Psi_\Theta(X) \in \sR^{N \times d_{\rm out}}$ represent the output of the network for some input $X \in \sR^{N \times d_{\rm in}}$. The network architecture is detailed in Appendix \ref{appendix:sec:distillation_experiments}. 
Here, we constrain all the weights to be orthonormal, thus posing the problem as optimization over the joint space of Stiefel and Euclidean manifolds. 
For experiments, we consider a six-layer network and set $d_{\rm in} = 500$, $d_{\rm out} = 200$. We use stochastic versions of RGD and RCD where the input samples are randomly generated. 
In Figure \ref{st_distill_fig}, RCD outperforms the baselines in terms of flop counts and runtime. This is because RCD has the most cost-efficient update per iteration, while maintaining a competitive convergence rate.

\begin{figure}[!t]
\centering
\subfloat{\includegraphics[width=0.23\textwidth]{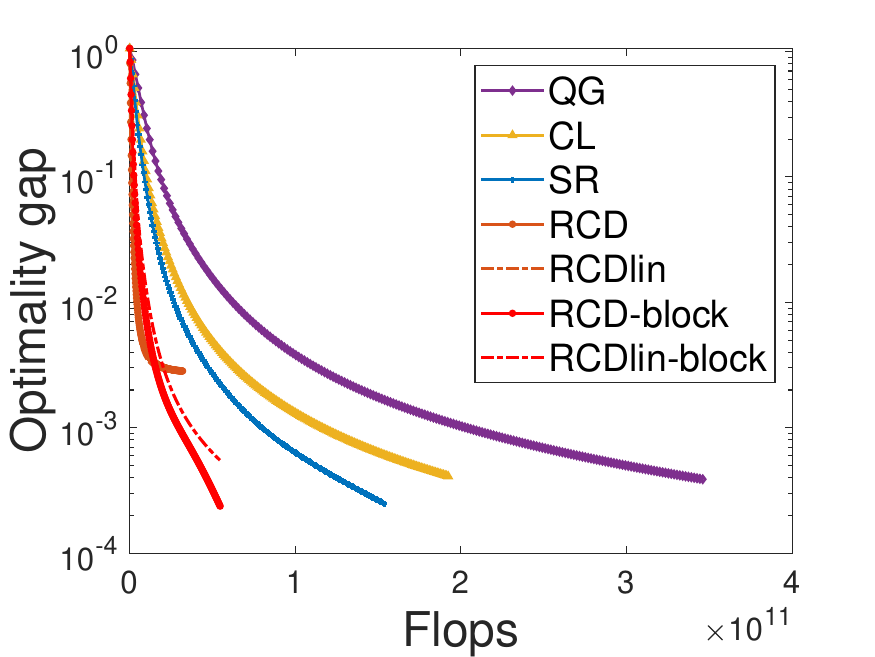}  }
\subfloat{\includegraphics[width=0.23\textwidth]{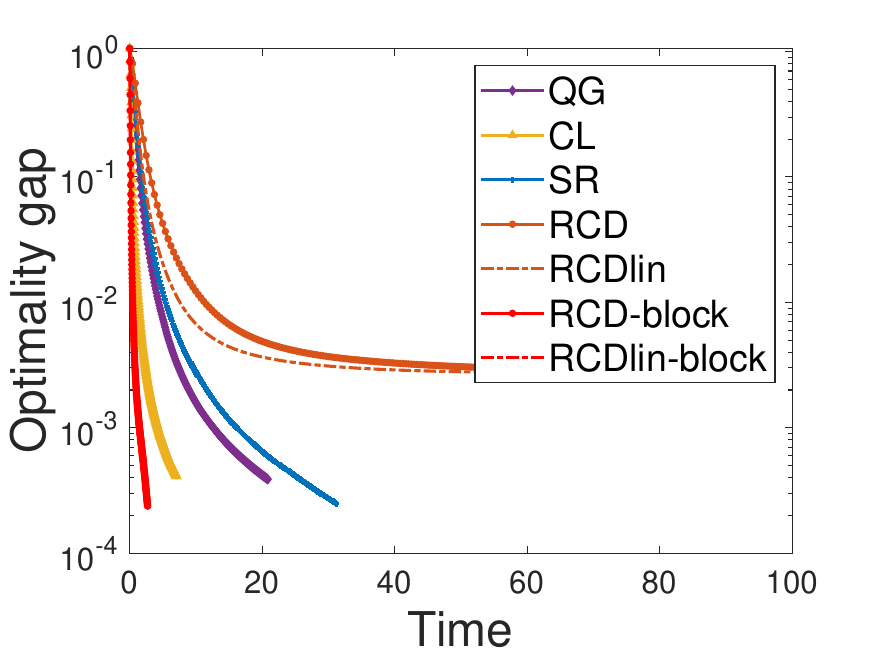}}
\caption{Experiments on the nearest matrix problem. We notice the utility of the block-update variants of our RCD and RCDlin algorithms in obtaining faster convergence.}
\label{symp_fig_embed}
\end{figure}

\subsection{Nearest matrix problem}



We consider the problem: $\min_{X \in \Sp(n,p)} \|X - A \|^2$ on the symplectic manifold \citep{gao2021riemannian}.
We follow the setting in \citep{gao2021riemannian} by generating $A$ as a random matrix with $p = n = 200$. The algorithms are evaluated on optimality gap $|f(X_k) - f^*|/|f^*|$, where $f^*$ is obtained by running the conjugate gradient algorithm with the Cayley retraction. We implement RCD and RCDlin with both CD and block CD updates (discussed in Remark \ref{rmk_block_symp}). As there is no CD baseline on the symplectic manifold, we compare against the full gradient RGD algorithms with three retractions: Cayley (`CL'), quasi-geodesic (`QG'), and SR (`SR'). 
In Figure \ref{symp_fig_embed}, RCD with block update shows clear advantage in both flop counts and runtime.

\subsection{Learning Lorentz embeddings} 

We consider the task of learning embeddings for word hierarchies, which is formulated on the hyperbolic manifold using the hyperboloid model \citep{nickel2017poincare,nickel2018learning,jawanpuria19b}. The goal is to map word pairs with hypernymy relations closer while separate those without. 
We follow the formulation in \citep{nickel2018learning}, and the details are in Appendix \ref{appendix:sec:lorentz_experiments}. 
%
%
%
\begin{figure}[t]
\centering
\subfloat{\includegraphics[width=0.21\textwidth]{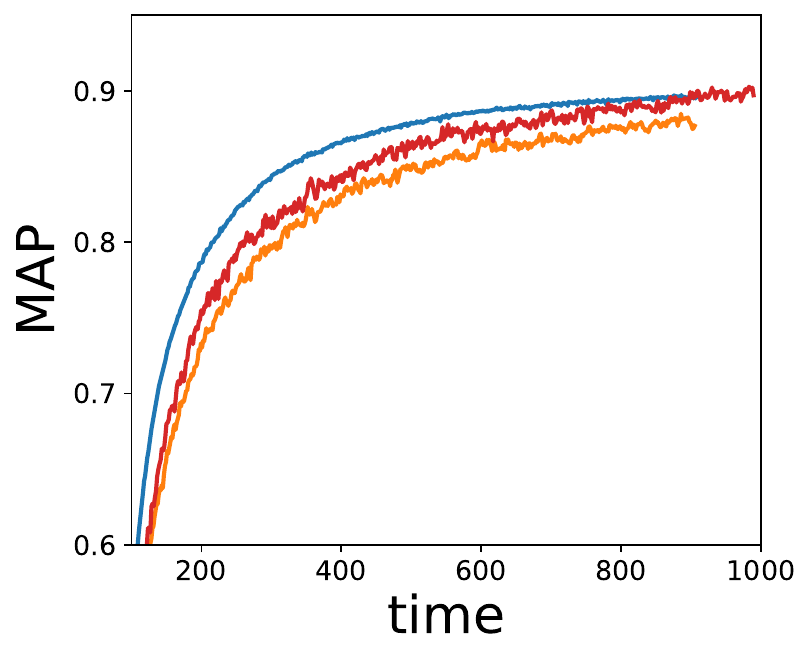}  } 
\hspace*{4pt}
\subfloat{\includegraphics[width=0.21\textwidth]{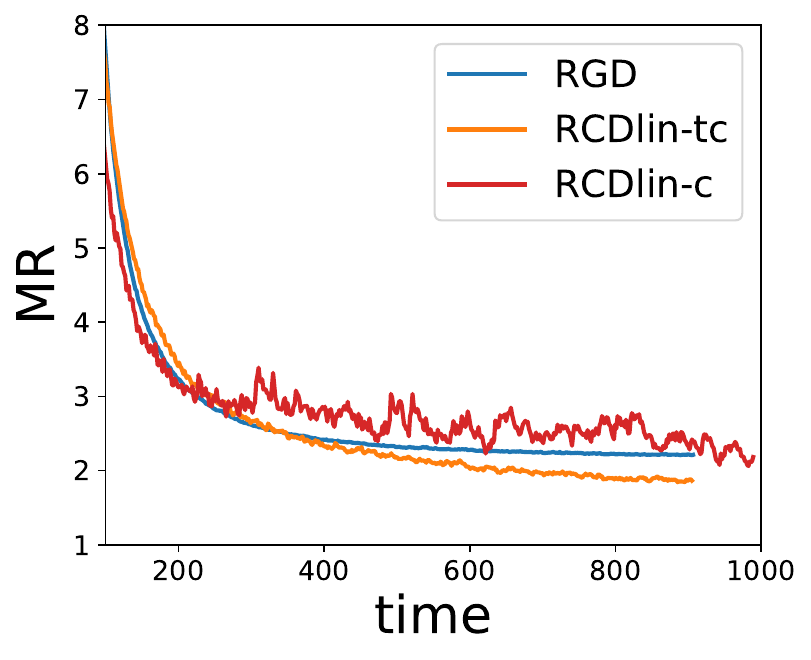}}
\caption{Experiments on learning Lorentz (hyperbolic) embeddings. The performance of our RCDlin algorithms (with cyclic and time-cyclic basis selection) is competitive to RGD. }
\label{hyper_fig_embed}
\end{figure}

For experiment settings, we train $5$-dimensional embeddings ($n = 5$) for WordNet mammals subtree \citep{miller1998wordnet}. 
We adopt the RCDlin algorithm with two selection rules for the basis $H_{ij} J X$: cyclic (`RCDlin-c') and time cyclic (`RCDlin-tc'). The cyclic selection loops through all $n(n-1)/2$ pairs per iteration. The time cyclic selection only loops through all the space-time coordinate pairs, namely $(1,2), (1,3), ..., (1, n)$, which reduces the computational cost to scale linearly with dimension $n$. For RCDlin-c and RCDlin-tc, we use a linearly-decaying stepsize, i.e., $\eta/(1+0.1\times {\rm epoch})$. For RGD we use a fixed stepsize $\eta$ which generally leads to better convergence. We tune and set $\eta = 1.0$ for RCDlin and $0.5$ for RGD. 
We use the metrics for evaluating the convergence: mean average precision (MAP) and mean rank (MR) \citep{nickel2017poincare,nickel2018learning}. 
In Figure \ref{hyper_fig_embed}, we see that RCDlin converges at a similar rate compared to RGD in terms of runtime.

\subsection{Weighted least squares (SPSD manifold)}


\begin{figure}[!t]
\centering
\subfloat[Optimization over SPD manifold with dense $A$]{\includegraphics[width=0.22\textwidth]{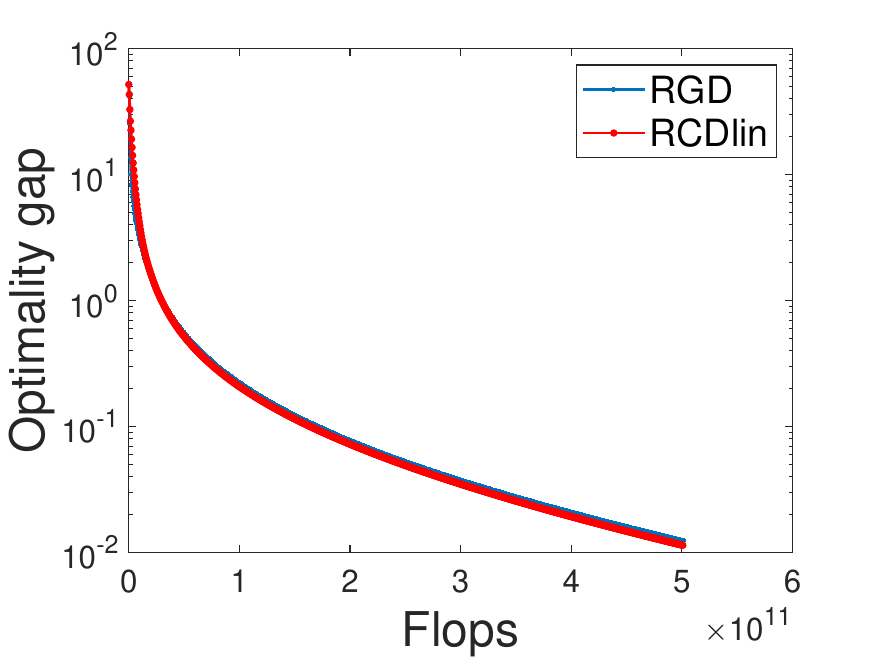}  }\ \ 
\subfloat[Optimization over SPSD manifold with sparse $A$]{\includegraphics[width=0.22\textwidth]{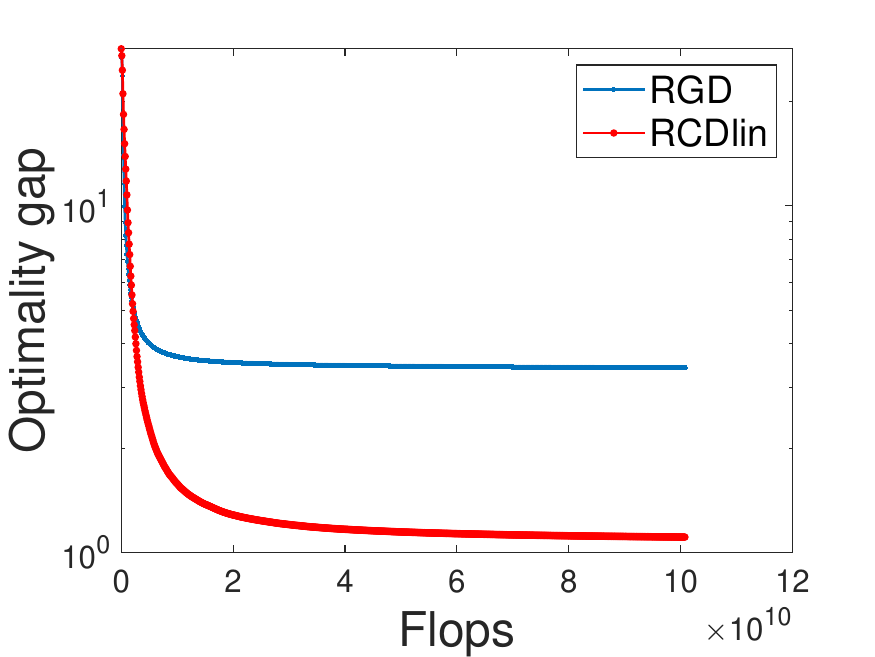}}
\caption{Experiments on the weighted least squares problem in two settings: (a) $n = p = 500$ and $A = 1_n 1_n^\top$ is a dense matrix and (b) $n = 500, p = 100$ and $A$ is a random symmetric matrix with $70\%$ entries as $1$ and others are $0$. While RGD and the proposed RCDlin have similar convergence rate in (a), RCDlin has clear advantage in (b).  }
\label{spsd_wls_figure}
\end{figure}

The weighted least squares problem is $\min_{X \in \sS_{+}^{n \times p}} \| A \odot X - B \|$, where $A \in \{ 0,1\}^{n \times n}$ masks the known entries in $B$. It is an instance of the matrix completion problem \citep{han2021riemannian}. 
For the experiment, we follow \citep{han2021riemannian} by generating $B = A \odot X^*$ where $X^*$ is an SPD/SPSD matrix with exponentially decaying eigenvalues. We consider two settings: (left) $n = p = 500$, $A = 1_n 1_n^\top$ is a dense matrix and (right) $n = 500, p = 100$ and $A$ is a random symmetric matrix with $70\%$ entries as $1$ and others are $0$. 
We compare RCDlin with RGD (for $X = YY^\top$ factorization). We set $S = np/5$ and select the coordinates randomly without replacement. In Figure \ref{spsd_wls_figure}, we observe that RCDlin performs competitively with RGD on the SPD manifold with dense $A$ while performing significantly better on the SPSD manifold with sparse $A$.




\section{Conclusion}

In this work, we discuss how to develop computationally efficient CD updates for a number of matrix manifolds. The main bottleneck in developing CD methods is on finding the right basis parameterization of the tangent space. We show precise constructions for various manifolds and propose two CD algorithms: RCD and RCDlin. RCDlin specifically reduces the gradient computations of RCD further. Our experiments show the benefit of our proposed CD updates on a number of problem instances.

\section*{Impact Statement}
This paper presents work whose goal is to advance optimization methods with applications in Machine Learning. There are many potential societal consequences of our work, none which we feel must be specifically highlighted here.

\bibliographystyle{icml2024}

\newpage
\appendix
\onecolumn

\section{Additional experiment details} 

The experiments are run on a laptop with an i5-10500 3.1GHz CPU processor.

\subsection{Orthogonal deep networks for distillation} \label{appendix:sec:distillation_experiments}

Here we provide the detailed network architecture for the distillation task.  
In particular, we define a $L$-layer feed-forward neural network with Tanh activation function, i.e.,
$X_{\ell+1} = \tanh(W_\ell X_{\ell} + b_\ell)$ for $\ell \in [L]$, where $W_1 \in \sR^{d_{\rm in} \times d}, W_L \in \sR^{d \times d_{\rm out}}$, and $W_\ell \in \sR^{d \times d}$ for $\ell \neq 1,L$. In the experiment, we set $L = 6$.

\subsection{Learning Lorentz embeddings} \label{appendix:sec:lorentz_experiments}
Here we provide the problem formulation for the task of learning Lorentz embeddings.
Let $\gD = \{(u,v) \}$ be the related word pairs and construct ${\rm Neg}(u) =\{ v : (u,v) \notin \gD \}$ as the negative samples of word $u$. The objective is to learn embeddings $x_u$ for all word $u$ by solving 
\begin{equation*}
    \min_{\{x_u \in \gH(n,1)\}_u} \sum_{(u,v) \in \gD} \log \frac{\exp \big( {-\dist(x_u, x_v)} \big)}{\sum_{v' \in {\rm Neg}(u)}  \exp \big( {-\dist(u, v')} \big)},
\end{equation*}
where $\dist(x_u, x_v) = {\rm arccosh}( - \langle x_u, x_v \rangle_\gL)$ is the Lorentz Riemannian distance.

\section{A review on coordinate descent for orthogonal and SPD manifold} \label{appendix:sec:related_works}

We start by reviewing the developments of coordinate descent on the orthogonal manifold \citep{shalit2014coordinate,massart2022coordinate,jianggivens}, Stiefel manifold \citep{gutman2023coordinate}, and symmetric positive definite (SPD) manifold \citep{darmwal2023low}, which motivate the proposed general framework for other manifolds.  

Some other works \citep{huang2021riemannian,peng2023block} study (block) coordinate descent on a product of manifolds, where each update concerns a component manifold. This is different to our considered setting, where the update is defined for coordinate on the tangent space for a single manifold.  

\subsection{Orthogonal manifold}

Orthogonal manifold $\gO(n)$ is the smooth space formed by the orthogonality constraints, i.e., $\gO(n) \coloneqq \{ X \in \sR^{n \times n}: XX^\top = X^\top X = I_n \}$. The tangent space can be identified as $T_X\gO(n) \coloneqq \{ \Omega X : \Omega \in \Skew(n) \}$. The Riemannian metric coincides with the Euclidean metric, i.e., $\langle U, V \rangle_X = \frac{1}{2} \langle U, V \rangle$. The $\frac{1}{2}$ is added to ensure consistency with the canonical metric for the Stiefel manifold, as we shall see later. This leads to the Riemannian gradient $\grad f(X) = \nabla f(X) - X \nabla f(X)^\top X$ and the exponential retraction is given by $\Retr_X(\theta \Omega X) = \expm(\theta \Omega) X$, for some $\theta \in \sR$. 

\begin{remark}
 We remark that in all the existing works \citep{shalit2014coordinate,massart2022coordinate,jianggivens}, the tangent space is parameterized as $T_X\gO(n) \coloneqq \{ X \Omega' : \Omega' \in \Skew(n) \}$ and thus the exponential retraction amounts to $\Retr_X(\theta X \Omega') = X \expm(\theta \Omega')$. Such a formulation is equivalent to the above by letting $\Omega = X \Omega' X^\top \in \Skew(n)$. Our reformulation allows natural generalization of the coordinate descent framework to column orthonormal matrices (the Stiefel manifold), where the orthogonal matrix is  a special case.
\end{remark}

The manifold has a dimension of $n(n-1)/2$ and its tangent space can be provided with an orthonormal basis $ H_{ij} X$ where $H_{ij} \coloneqq e_i e_j^\top - e_j e_i^\top$ is the basis for the skew-symmetric matrices. In each basis direction $H_{ij} X$, the exponential retraction reduces to the Givens rotation, which allows efficient updates, i.e., $ \Retr_X(\theta H_{ij} X) = G_{ij}(\theta) X$ where $G_{ij}(\theta) = I_n + (\cos(\theta) - 1) (e_ie_i^\top + e_je_j^\top) + \sin(\theta) (e_ie_j^\top - e_j e_i^\top)$ is known as the Givens rotation matrix around axes $i, j$ with angle $-\theta$. 

In order to minimize a function $f: \gO(n) \rightarrow \sR$, one needs to update the variables in the negative gradient direction. Here along the basis direction, coordinate descent aims to minimize the function $f(G_{ij}(\theta) X)$ with respect to $\theta$. One strategy is to solve this one-variable optimization problem directly as in \citep{shalit2014coordinate}. When the objective is more involved, we can approximately solve this problem by following a descent direction \citep{massart2022coordinate,jianggivens}, which is given by $-\frac{d}{d\theta} f(G_{ij}(\theta) X) \vert_{\theta=0} = -\langle \grad f(X), H_{ij} X\rangle_X$.
This leads to the coordinate descent update in the direction of $-\langle \grad f(X), H_{ij} X\rangle_X H_{ij} X$, which modifies two rows of $X$ every iteration, resulting in an $O(n)$ complexity per update. One pass over all the coordinates requires $\frac{n(n-1)}{2}$ iterations, leading to $O(n^3)$ complexity in total, which is comparable to the commonly considered retractions, including the exponential, Cayley and QR retractions.

\subsection{Stiefel manifold}
The Stiefel manifold $\St(n,p)$ is the set of column orthonormal matrices of size $\sR^{n \times p}$, i.e., $\St(n,p) \coloneqq \{ X \in \sR^{n \times p} : X^\top X = I_p \}$. When $p = n$, $\St(n,n) \equiv \gO(n)$. The tangent space of Stiefel manifold is identified as $T_X\St(n,p) = \{ U \in \sR^{n \times p} : X^\top U + U^\top X = 0 \} = \{ X \Omega + X_\perp K : \Omega \in \Skew(p), K \in \sR^{(n-p)\times p} \}$. The Euclidean metric turns out to be a valid Riemannian metric \citep{edelman1998geometry}, $\langle U, V \rangle_X = \langle U, V\rangle$ for any $U, V \in T_X\M$. The Riemannian gradient is derived as  $\grad f(X) = \nabla f(X) - X \sym(X^\top \nabla f(X))$. 

In \citep{gutman2023coordinate}, a coordinate (subspace) descent algorithm has been developed for general manifolds, via selecting proper subspaces for the tangent space. Although showing theoretical guarantees, the paper only provides a concrete developments for Stiefel manifold (thus including the orthogonal manifold). The basis considered for the tangent space of Stiefel manifold is $\{ X H_{ij}  \}_{1\leq i < j \leq n} \cup \{ v e_k^\top : X^\top v = 0 \}_{k \in [p]}$, where $H_{ij} = e_i e_j^\top - e_j e_i^\top$. They show along the direction of basis, the exponential retraction can be computed efficiently. That is, for basis $X H_{ij}$, we compute $\Retr_X(t XH_{ij}) = X G_{ij}(t)$, which also leads to the Givens rotation. The projection of Riemannian gradient onto the basis is given by $\langle \grad f(X), X H_{ij} \rangle X H_{ij}$. For basis $v e_k^\top$, $\Retr_X( v e_k^\top) = X + \big(\cos(\| v\|) [X]_{:,j} + \sin(\| v\|) \frac{v}{\| v\|} -  [X]_{:,j} \big) e_k^\top$ and the projection of Riemannian gradient is $(I_n - XX^\top) [\nabla f(X)]_{:,k} e_k^\top$. This essentially updates $j$-th column by Riemannian gradient descent over sphere.
We notice that each coordinate update costs $O(n)$ while the projection onto second basis costs $O(np)$.

Further, we highlight a recent paper \citep{yuan2023block}, which considers block coordinate descent updates for Stiefel manifold by modifying $k$ rows. The strategy is to decompose the the rows of the variable into two sets and solve for a subproblem that updates $k$ rows instead. Nevertheless, each subproblem can be difficult to solve in general. 


\subsection{SPD manifold}

A recent work \citep{darmwal2023low} develops coordinate (subspace) descent algorithm for symmetric positive definite (SPD) manifold, i.e., $\sS_{++}^n \coloneqq \{ X \in \sR^{n \times n} : X \succ 0\}$, with tangent space given by $T_X\sS_{++}^n = \{ X \in \sR^{n \times n} : X \in \Sym(n) \}$. Optimization with SPD constraint is difficult as the cost of maintaining positive definiteness requires at least $O(n^3)$. Under the affine-invariant metric \citep{bhatia2009positive}, i.e., $\langle U, V\rangle_X = \trace(U X^{-1} V X^{-1})$, the paper verifies that the tangent space can be provided with an orthonormal basis $L {E}_{ij} L^\top$ where $LL^\top = X$ is the Cholesky decomposition and ${E}_{ij} = (e_i e_j^\top + e_j e_i^\top)/\sqrt{2}$ for $i \neq j$ and $e_i e_j^\top$ for $i = j$.
The exponential retraction along the basis direction can be simplified as $\Retr_X(t L {E}_{ij} L^\top) = L \expm(t {E}_{ij}) L^\top$, where the Cholesky factor of $\expm(t {E}_{ij})$ has a simple form. Thus it suffices to update the Cholesky factor. The coordinate descent then updates as $\Retr_X(-\eta \theta L {E}_{ij} L^\top )$ where $\theta$ is computed as the coordinate derivative $\theta = \langle \grad f(X), L {E}_{ij} L^\top \rangle_X = \langle \nabla f(X), L {E}_{ij} L^\top \rangle$ because the Riemannian gradient has the form $\grad f(X) = X \nabla f(X) X$.

\section{Other related works}
This section summarizes the existing works for optimization under the respective manifold constraints. These methods exploit the full gradient information and in general have advantage over CD methods.

\paragraph{Optimization on orthogonal, Stiefel, and Grassmann manifolds.}
Optimization with orthogonality constraint has been widely studied. Apart from the Riemannian optimization approach, many recent works turn to infeasible methods, either through converting the orthogonality constrained problem into an unconstrained counterpart in the Euclidean space \citep{xiao2022class,lezcano2019trivializations,xiao2021solving,xiao2023dissolving}, or following a direction that leads to the same critical point on the manifold \citep{gao2019parallelizable,ablin2022fast,ablin2023infeasible}. We provide a through review of these methods in Appendix section \ref{appe_sect_ortho}. Although the per-iteration cost is smaller without using a retraction to ensure feasibility, the methods are sensitive to the choice of stepsize and other parameters, mostly a regularization parameter. In the experiment sections, we observe the infeasible methods either require careful tuning of multiple parameters or some stepsize sequence to show good performance.

\paragraph{Optimization on hyperbolic, symplectic and SPSD manifold.}
Optimization over hyperbolic space $\gH(n,p)$ has mostly focused on the case where $p = 1$ \citep{nickel2018learning,wilson2018gradient}. Given the exponential map is already efficient in this case, few works have explored more efficient alternatives to Riemannian optimization. Similarly for the symplectic manifold, existing works \citep{gao2021riemannian,gao2021geometry,gao2022optimization} have focused on the retraction-based Riemannian optimization approach. Optimization over SPD matrices has a long history and can be solved with semidefinite programming if the objective is convex. For general cost functions and with additional rank constraint, many works \citep{vandereycken2009embedded,vandereycken2013riemannian,massart2020quotient} have defined Riemannian geometries for the constraint set to leverage the tools from Riemannian optimization.

\section{Review of infeasible methods for optimization under orthogonality constraints}
\label{appe_sect_ortho}
This section reviews various infeasible methods for solving optimization problems under orthogonality constraints, i.e., 
\begin{equation}
    \min_{X} f(X), \quad \text{s.t.}  X^\top X = I_p. \label{main_ortho}
\end{equation}
Recall the {first-order optimality conditions} of the problem are 
$$\begin{cases}
    \nabla f(X) = X\Lambda, \\
    {X}^\top X = I_p
\end{cases}$$
where $\Lambda = \Lambda^\top$ is the Lagrangian multipliers. Because $\Lambda$ is symmetric, we can compute $\Lambda = X^\top \nabla f(X) = \nabla f(X)^\top X = \sym(\nabla f(X)^\top X)$ (by left-multiplying the first equation by $X^\top$). Subsequently, the first condition becomes  $\nabla f(X) - X \nabla f(X)^\top X = \skew(\nabla f(X) X^\top) X = 0$, which in fact corresponds to the first-order condition of Riemannian optimization (under the canonical metric). 
The augmented Lagrangian of \eqref{main_ortho} is given by
\begin{equation*}
    \gL_\mu (X, \Lambda) = f(X) - \frac{1}{2} \langle \Lambda, X^\top X - I_p \rangle + \frac{\mu}{4} \| X^\top X - I_p \|^2
\end{equation*}
and the Augmented Lagrangian Method (ALM) outlined in \citep{gao2019parallelizable} considers alternately updating $X$ and $\Lambda$. However the numerical performance of ALM is sensitive to the choice of $\mu$, which has been empirically verified in \citep{gao2019parallelizable}.


\paragraph{PLAM.} \citet{gao2019parallelizable} propose an alternative update of the Lagrangian multipliers $\Lambda$ by setting it to  $\sym(\nabla f(X)^\top X)$, which is optimal at first-order stationarity, and the proximal linearized augmented Lagrangian (PLAM) algorithm takes the update 
\begin{equation}
    X_{k+1} = X_k - \eta \Big( \nabla f(X_k) - X_k \sym(\nabla f(X_k)^\top X_k) + \lambda X_k (X_k^\top X_k - I_p)  \Big), \label{rejrkem}
\end{equation}
which corresponds to the $X_{k+1} = \argmin_{X} \big\langle \nabla_X \gL_\mu(X_k, \Lambda_k ), X- X_k \big\rangle  + \frac{1}{2\eta} \| X - X_k\|^2$ with $\Lambda_k = \sym(\nabla f(X_k)^\top X_k)$. However,
the boundedness of the iterates cannot be expected without setting $\lambda$ to be sufficiently large and $\eta$ to be sufficiently small. 

\paragraph{PCAL.} \citet{gao2019parallelizable} further constrain the proximal linearized update in \eqref{rejrkem} to a redundant column-sphere constraint. This is used to reduce the sensitivity of convergence to $\lambda, \mu$. The update can be implemented in a column-wise parallel fashion as 
\begin{equation*}
    (X_{k+1})_i = \frac{(X_k)_i - \eta \nabla_{X_i} \gL_\mu(X_k, \Lambda_k)}{\| (X_k)_i - \eta \nabla_{X_i} \gL_\mu(X_k, \Lambda_k) \|}.
\end{equation*}

\paragraph{PenCF} \citet{xiao2022class} consider the merit function (used to analyze function value decrease in \citep{gao2019parallelizable}) as an exact penalty function to be minimized. 
Specifically, the merit function is given by 
\begin{equation*}
    h(X) = f(X) - \frac{1}{2} \langle \sym(X^\top \nabla f(X)), X^\top X - I_p \rangle + \frac{\beta}{4} \| X^\top X - I_p \|.
\end{equation*}
Without proper constraints, directly minimizing $h(X)$ can be problematic because $h(X)$ can be unbounded. Hence, \citep{xiao2022class} incorporates constraints to the minimization of $h(X)$ (called PenC), including a ball, convex hull of Stiefel/oblique manifold, which includes the orthogonality constraint set. Then the paper shows the first-order and second-order critical points of PenC match those of the original problem (provided $\beta$ is sufficiently large). 
Nevertheless, the gradient $\nabla h(X)$ involves second-order derivatives $\nabla^2 f(X)$. Hence, the paper similarly considers approximating $\nabla h(X)$ with $\nabla f(X) - X \sym(\nabla f(X)^\top X) + \beta X(X^\top X - I_p)$, which is the same as in \citep{gao2019parallelizable} and then project the update to the constraint set.


\paragraph{Landing field} \citet{ablin2023infeasible} consider the update 
\begin{align*}
    X_{k+1} = X_k - \eta \Big( \skew(\nabla f(X_k)X_k^\top) X_k + \lambda X_k (X_k^\top X_k - I_p) \Big),
\end{align*}
where the direction takes a combination of the Riemannian gradient and a landing direction orthogonal to the Riemannian gradient. 

\paragraph{ExPen} \citet{xiao2021solving,xiao2023dissolving} convert the constrained optimization problem to an unconstrained problem by minimizing the following objective: 
\begin{equation*}
    h(X) = f \Big( X (\frac{3}{2} I_P - \frac{1}{2} X^\top X) \Big) + \frac{\beta}{4} \| X^\top X - I_p\|^2.
\end{equation*}
It can be shown that when $\beta$ is sufficiently large, first-order and second-order critical points recover those of the original problem. When compared to the framework of PCAL/PenC, the problem is now unconstrained and the gradient can be computed as 
\begin{equation*}
    \nabla h(X) = G(X) (\frac{3}{2} I_p - \frac{1}{2} X^\top X) - X \sym(X^\top G(X)) + \beta X(X^\top X - I_p)
\end{equation*}
where $G(X) = \nabla f(Y) \vert_{Y = X(\frac{3}{2} I_p - \frac{1}{2} X^\top X)}$. This allows unconstrained solvers to be directly applied.


\section{Additional developments and proofs for Section \ref{main_sect_cd}}

\subsection{Stiefel manifold}

In the main text, we focus on the Euclidean metric on Stiefel manifold, i.e., $\langle U, V \rangle_X^E = \langle U, V\rangle$, which leads to the Riemannian gradient $\grad f(X) = \nabla f(X) - X \sym(X^\top \nabla f(X))$.
We here review another popular metric on Stiefel manifold, namely the canonical metric \citep{edelman1998geometry}, defined as $\langle U, V \rangle^C_X = \langle U, (I_n - \frac{1}{2} XX^\top) V\rangle$. The corresponding Riemannian gradient is derived as $\grad_C f(X) = \nabla f(X) - X \nabla f(X)^\top X$. 

There exists a variety of retractions for the Stiefel manifold, including the QR retraction $\Retr_X^{\rm qr}(t U) ={\rm qf}(X + t U)$ where ${\rm qf}(A)$ extracts the q-factor from the thin QR decomposition. The Cayley retraction \citep{wen2013feasible,zhu2017riemannian} is 
$\Retr_X^{\rm cl} (t U) = (I - \frac{t}{2} W)^{-1} (I + \frac{t}{2} W) X$, where $U = WX$ for some $W \in \Skew(n)$. The exponential retraction is given by 
$ \Retr_X^{\rm exp}(t U) = \begin{bmatrix}
        X  & U
    \end{bmatrix}
    \expm \Big( \begin{bmatrix}
        X^\top U & - U^\top U \\
        I & X^\top U
    \end{bmatrix} \Big)
    \begin{bmatrix}
        \expm(-X^\top U) \\
        0
    \end{bmatrix}.
$

Here we verify the same coordinate derivative is recovered under the canonical metric.
\begin{lemma}
\label{lem_theta_stiefel}
Let $B_\ell = H_{ij} X$ for $1\leq i < j \leq n$, then we have $\theta = \langle \grad_E f(X), B_\ell \rangle^E_X = \langle \grad_C f(X), B_\ell \rangle_X^C = \langle \nabla f(X), B_\ell \rangle = [\nabla f(X) X^\top - X \nabla f(X)^\top ]_{ij}$.
\end{lemma}
\begin{proof}[Proof of Lemma \ref{lem_theta_stiefel}]
Recall the $\grad_E f(X) = \nabla f(X) - X\sym(X^\top \nabla f(X))$. Then
\begin{equation*}
    \theta = \langle \grad_E f(X), B_\ell \rangle_X^E = \langle \nabla f(X) - X \sym(X^\top \nabla f(X)), H_{ij} X\rangle = \langle \nabla f(X), H_{ij} X \rangle = 2[\skew(\nabla f(X) X^\top)]_{ij},
\end{equation*}
where we use the fact that skew-symmetric matrix is orthogonal to symmetric matrix with respect to the Euclidean inner product.
Similarly for the canonical metric,
\begin{align*}
    \theta = \langle \grad_C f(X), B_\ell \rangle_X^C &= \langle \nabla f(X) - X \nabla f(X) X^\top , (I_n - \frac{1}{2} XX^\top) H_{ij} X\rangle \\
    &= \langle \nabla f(X) - X \sym(X^\top \nabla f(X)), H_{ij} X\rangle \\
    &= \langle \nabla f(X), H_{ij} X \rangle.
\end{align*}
This verifies $\theta$ is the same under the two metrics.
\end{proof}

\subsection{Grassmann manifold}

\begin{proof}[Proof of Proposition \ref{grass_cd_prop}]
 It suffices to verify that $\nabla f(XQ) = \nabla f(X) Q$ as the Euclidean inner product $\langle \nabla f(X) Q, H_{ij} XQ \rangle = \langle \nabla f(X), H_{ij} X \rangle$. We start by noticing $f(X) = f(XQ)$ for any $Q \in \gO(p)$. Then taking derivative on both sides gives $\nabla f(X) = \nabla f(XQ) Q^\top$ and thus $\nabla f(XQ) = \nabla f(X) Q$. 
\end{proof}

\subsection{Hyperbolic manifold} \label{appendix:sec:hyperbolic}

\subsubsection{Proofs}

\begin{proof}[Proof of Proposition \ref{prop_hyper_ortho_grad}]
By the decomposition, we can express $A = X\Omega + X_{J_\perp} K + XS$ for some $\Omega \in \Skew(p), S \in \Sym(p), K \in \sR^{(n-p) \times p}$. Left-multiplying $X^\top J$ gives $X^\top J A = - \Omega - S$. Summing both sides with the transposes yields $S = - \sym(X^\top JA)$. Hence the projection to $T_X\gH(d,r)$ is given by ${\rm Proj}_X(A) = A + X\sym(X^\top J A)$. 

From the definition of Riemannian gradient, we have $\D f(X) [U] = \langle \nabla f(X), U \rangle_F = \langle J \nabla f(X), U \rangle_\gL = \langle J \nabla f(X), {\rm Proj}_X(U)  \rangle_\gL = \langle {\rm Proj}_X(J\nabla f(X)),  U\rangle_\gL$, where we use the self-adjoint property of orthogonal projection with respect to the metric. Thus the Riemannian gradient is $\grad f(X) = {\rm Proj}_X(J \nabla f(X)) = J \nabla f(X) + X \sym(X^\top \nabla f(X))$. 
\end{proof}

\begin{proof}[Proof of Proposition \ref{prop_expm_retr_hyper}]
First, we see $c(0) = X$ and we compute $c'(t) = \expm(t WJ) WJX$ and hence $c'(0) = WJX = U$. The only part left is to show $c(t) \in \gH(n,p)$. 
This can be verified by showing $\expm(tWJ)$ is a Lorentz transform. To see this, let $L(t) \coloneqq \expm(tWJ)^\top J \expm(tWJ)$. Then we have 
\begin{align*}
    L'(t) &= (WJ \expm(tWJ))^\top J \expm(tWJ) + \expm(tWJ)^\top J WJ \expm(tWJ)  \\
    &= \expm(tWJ)^\top (JW^\top J + JWJ) \expm(tWJ) = 0, \quad \forall t.
\end{align*}
Thus $L(t) = L(0) = J, \forall t$.  This completes the proof as $c(t) = \expm(tWJ) X \in \gH(n,p)$ because $c(t)^\top J c(t) = X^\top L(t) X = X^\top J X = -I_p$. 
\end{proof}

\begin{proof}[Proof of Lemma \ref{lemma_eij_lorentz}]
For the case where $i, j \neq 1$, we can see $H_{ij} J = H_{ij}$ and thus $\expm(\theta H_{ij})$ leads to the Givens rotation. However in the case where $i = 1$, $H_{ij} J = E_{ij} \coloneqq e_ie_j^\top + e_je_i^\top$. Thus, it remains to show $\expm(\theta E_{ij})$ can be simplified.
To this end, we see $E_{ij}^{2t} = \frac{1}{2} (E_{ii} + E_{jj})$, $E_{ij}^{2t-1} = E_{ij}, \text{ for } t \in \mathbb{N}$. Then we can show $\expm(\theta E_{ij}) = I + \theta E_{ij} + \frac{1}{2!} \theta^2 E_{ij}^2 + \frac{1}{3!} \theta^3 E_{ij}^3 +\cdots = \frac{1}{2} \sum_{k \neq i,j} E_{kk} + (1 + \frac{\theta^2}{2!} + \frac{\theta^4}{4!} + \cdots) \frac{1}{2} (E_{ii} + E_{jj}) + (\theta + \frac{\theta^3}{3!} + \frac{\theta^5}{5!} + \cdots )E_{ij} =  \frac{1}{2} \sum_{k \neq i,j} E_{kk} +\frac{\cosh(\theta)}{2}(E_{ii} + E_{jj}) + \sinh(\theta) E_{ij}$. 
\end{proof}

\subsubsection{A canonical-type metric}
In addition to the Euclidean metric, we define the canonical metric on hyperbolic manifold as for $U, V \in T_X\gH(n,p)$,
\begin{align*}
    \langle U, V \rangle_X = -\trace \big(U^\top (J + \frac{1}{2} JXX^\top J) V \big) &= -\trace \big( (X\Omega_u + X_{J_\perp} K_u)^\top (J + \frac{1}{2} J XX^\top J) (X\Omega_v + X_{J_\perp} K_v)  \big) \\
    &= \frac{1}{2}\trace(\Omega_u^\top \Omega_v) + \trace(K_u^\top K_v).
\end{align*}
The normal space under the canonical metric is the same as that for the Euclidean metric  $N_X\gH(d,r) = \{ X S : S \in \Sym(p) \}$ and thus the orthogonal projection can be derived also to be the same. The Riemannian gradient can be derived as follows.

\begin{proposition}
\label{prop_can_ortho_grad}
The orthogonal projection of $A \in \sR^{n \times p}$ to $T_X \gH(n,p)$ is given by ${\rm Proj}_X(A) = A + X \sym(X^\top JA)$. The Riemannian gradient with respect to the canonical metric is $\grad f(X) =- J \nabla f(X) - X \nabla f(X)^\top X = 2 \skew(J \nabla f(X) X^\top) JX$.
\end{proposition}
\begin{proof}[Proof of Proposition \ref{prop_can_ortho_grad}]
First we notice that $(J + \frac{1}{2} JXX^\top J)^{-1} = J - XX^\top$. This implies $\langle U, V \rangle = \trace(U^\top V) = \trace(U^\top (J - XX^\top) (J + \frac{1}{2} JXX^\top J) V ) = - \langle (J - XX^\top) U , V\rangle_X$.
Then by definition of Riemannian gradient, we have $\langle \nabla f(X), U \rangle = -\langle (J - XX^\top) \nabla f(X), U \rangle_X = \langle {\rm Proj}_X \big( (XX^\top - J) \nabla f(X) \big) , U\rangle_X.
$    Hence, $\grad f(X) = {\rm Proj}_X\big( (XX^\top - J) \nabla f(X) \big) = (XX^\top - J)\nabla f(X) = - J \nabla f(X) - X \nabla f(X)^\top X$.
\end{proof}

\subsubsection{Cayley retraction}
Motivated by the Cayley transform for (generalized) Stiefel manifold, we define the Cayley transform for generalized hyperbolic manifold as follows. Then in Proposition \ref{hyper_prop_cayley}, we show the Cayley transform naturally leads to a valid retraction on the generalized hyperbolic manifold.
\begin{definition}
For $X \in \gH(d,r)$ and $U = WJX \in T_X \gH(d,r)$, the Cayley transform is defined as ${\rm Cay}_X(U) = (I - \frac{1}{2} WJ)^{-1} (I + \frac{1}{2} WJ) X$, which is well-defined if $I - \frac{1}{2} WJ$ is nonsingular.
\end{definition}

\begin{proposition}[Cayley retraction]
\label{hyper_prop_cayley}
For $X \in \gH(d,r)$ and $U \in T_X \gH(d,r)$, The map $\Retr^{C}_X(tU) \coloneqq {\rm Cay}_X(tU) = (I - \frac{t}{2} W_U J)^{-1} (I + \frac{t}{2} W_UJ)X$ is a retraction, where $W_U = XU^\top P_X - P_X^\top UX^\top$, with $P_X = I_n + \frac{1}{2} JXX^\top$.
\end{proposition}
\begin{proof}[Proof of Proposition \ref{hyper_prop_cayley}]
First, because $U \in T_X\gH(d,r)$, we can express $U = W JX$ for some $W \in \Skew(n)$. Since $\gH(d,r)$ is an embedded submanifold of $\sR^{d\times r}$, let the curve $c(t) =\Retr^C_X(tU)$ and we have $c(0) = X$ and 
\begin{align*}
    c'(t) &= \frac{1}{2} (I - \frac{t}{2} WJ)^{-1}  WJ (I - \frac{t}{2} WJ)^{-1} (I + \frac{t}{2} WJ) X + \frac{1}{2} (I - \frac{t}{2} WJ)^{-1} WJ X \\
    &= \frac{1}{2} (I - \frac{t}{2} WJ)^{-1} WJ (X + c(t))
\end{align*}
with $c'(0) = WJ X = U$. It remains to show $c(t) \in \gH(d,r)$. First we notice that 
\begin{equation*}
    c(t) = (I - \frac{t}{2} WJ)^{-1} (I + \frac{t}{2} WJ) X = (J - \frac{t}{2} JWJ)^{-1} (J + \frac{t}{2} JWJ)X.
\end{equation*}
Then 
\begin{align*}
    c(t)^\top J c(t) &= X^\top (J + \frac{t}{2} JW^\top J) (J - \frac{t}{2} JW^\top J)^{-1} J (J - \frac{t}{2} JWJ)^{-1} (J + \frac{t}{2} JWJ)X \\
    &= X^\top (J - \frac{t}{2} J WJ) (J + \frac{t}{2} JWJ)^{-1} J (J - \frac{t}{2} JWJ)^{-1} (J + \frac{t}{2} JWJ)X \\
    &= X^\top (I - \frac{t}{2} JW ) (I + \frac{t}{2} JW)^{-1} (I - \frac{t}{2} JW)^{-1} (I + \frac{t}{2} JW) JX \\
    &= X^\top JX = -I_r, 
\end{align*}
where the second equality uses the fact that $W \in \Skew(d)$ and the last equality is due to $(I + A)(I-A) = (I-A) (I+A)$ for any $A$.

Finally we can verify that $W_U = XU^\top P_X- P_X^\top UX^\top$ where $P_X = I_n + \frac{1}{2}JXX^\top$ satisfies $W_U JX = U$. That is, 
\begin{equation*}
    W_U J X = \frac{1}{2} XU^\top JX + (I_n + \frac{1}{2} XX^\top J) U = U + \frac{1}{2} X (U^\top JX + X^\top J U) = U
\end{equation*}
where we use $U \in T_X \gH(n,p)$. 
\end{proof}

\subsection{Symplectic manifold}

\subsubsection{Review of canonical metric and retractions}
The canonical metric of symplectic manifold is developed in \citep{gao2021riemannian}, as $\langle U, V\rangle_X = \frac{1}{\rho} \langle S_u, S_v \rangle + \langle K_u, K_v \rangle$ for some chosen $\rho > 0$, where $U = X \Omega_p S_u + \Omega_n X_\perp K_u$ and $V = X\Omega_pS_v + \Omega_n X_\perp K_v$. The Riemannian gradient associated with the metric is $\grad_\rho f(X) = \rho X \Omega_p \sym(\Omega_p^\top X^\top \nabla f(X)) + \Omega_n X_\perp X_\perp^\top \Omega_n^\top \nabla f(X)$. The quasi-geodesic retraction is derived by replacing the covariance derivative with the Euclidean derivative, given by $\Retr_X^{\rm qgeo}(tU) = [X, U] \expm\Big( t \begin{bmatrix}
    -\Omega_p W & \Omega_p U^\top \Omega_n U \\
    I_{2p} & - \Omega_p W
\end{bmatrix} \Big) \begin{bmatrix}
    \expm(t \Omega_n W) \\ 0
\end{bmatrix}$ where $W = X^\top \Omega_n U$. The symplectic Cayley retraction is derived to be $\Retr^{\rm cay}_X(tU) = \Big( I_{2n} - \frac{t}{2} S_{X, U} \Omega_n \Big)^{-1} \Big( I_{2n} + \frac{t}{2} S_{X, U} \Omega_n\Big) X$ where $S_{X,U} = G_X U (X \Omega_p)^\top + X \Omega_{2p} (G_XU)^\top$, $G_X = I_{2n} - \frac{1}{2} X \omega_p X^\top \Omega_n^\top$. In \citep{gao2022optimization}, a SR decomposition based retraction is proposed. That is, let $P_{2p} \coloneqq [e_1, e_3, ..., e_{2p-1}, e_2, ..., e_{2p}]$ where $e_j$ is the $j$-th basis vector of $\sR^{2p}$. Then denote a congruence matrix set as $T_{sk}(P_{2p}) \coloneqq \{ P_{2p}^\top \hat{R} P_{2p} : \hat{R} \in \sR^{2p \times 2p} \text{ is upper triangular} \}$. Then the $\Retr^{\rm sr}_X(tU) = {\rm sf}(X + t U)$ where $A = S R$ is the SR decomposition of $A \in \sR^{2n \times 2p}$, with $S \in \Sp(2n,2p)$ and $R \in T_{2p}(P_{2p})$ and ${\rm sf}(A)$ extracts the $S$ factor.

\subsubsection{Proofs}

The proof of Proposition \ref{prop_exp_retr} follows immediately from the following Lemma.

\begin{lemma}
\label{lem_sp_group}
For any $S \in \Sym(2n)$, we have $\expm(t\Omega_n S), \expm(tS \Omega_n) \in \Sp(n,n)$. 
\end{lemma}

\begin{proof}[Proof of Lemma \ref{lem_sp_group}]
The proof follows from \citep[Proposition 4.6]{gao2021riemannian}. 
\end{proof}

\begin{proof}[Proof of Proposition \ref{prop_exp_retr}]
Similar to the previous sections, let $c(t) \coloneqq \Retr_X(tU)$ and we have $c(0) = X$ and $c'(0) = S \Omega_n X = U$. Finally from Lemma \ref{lem_sp_group}, we have $c(t)^\top \Omega_n c(t) = X^\top \expm(tS\Omega_n)^\top \Omega_n \expm(tS \Omega_n) X = X^\top \Omega_n X = \Omega_p$, which verifies $c(t) \in \Sp(n,p)$   
\end{proof}

\begin{proof}[Proof of Proposition \ref{prop_retr_sympl_cd}]
We can partition the basis $E_{ij} = \begin{bmatrix}
    E_{ij,1} & E_{ij,2} \\
    E_{ij,2}^\top & E_{ij,3}
\end{bmatrix} \in \{ 0,1\}^{2n \times 2n}$, where $E_{ij,1}, E_{ij,3} \in \Sym(n)$ and $E_{ij,2} \in \{0,1\}^{n \times n}$. Hence $E_{ij} \Omega_n = \begin{bmatrix}
    -E_{ij,2} & E_{ij,1} \\
    -E_{ij,3} &  E_{ij,2}^\top
\end{bmatrix}$ and the aim is to express $\expm(\theta E_{ij} \Omega_n)$ in compact form. 

For $1\leq i \leq j \leq n$, we have $E_{ij,2}, E_{ij,3} = 0$ and thus we obtain $\exp(\theta E_{ij} \Omega_n) = I + \theta E_{ij}\Omega_n + \frac{\theta^2}{2} (E_{ij} \Omega_n)^2 + \cdots = I + \theta E_{ij}\Omega_n$. Similarly for $n+1 \leq i < j \leq 2n$, we have $E_{ij,1}, E_{ij,2} = 0$ and $\expm(\theta E_{ij} \Omega_n) = I + \theta E_{ij} \Omega_n$. This verifies $\forall 1\leq i \leq j \leq n \text{ or } n+1 \leq i < j \leq 2n$
\begin{equation*}
    \Retr_X(\theta E_{ij} \Omega_n X) = X + \theta E_{ij}\Omega_n X.
\end{equation*}

For $1\leq i \leq n < j \leq 2n$, we have $E_{ij,1} = E_{ij,3} = 0$ and $E_{ij,2} = e_{i} e_{j - n}^\top$. Then $E_{ij} \Omega_n = \begin{bmatrix}
    - E_{ij,2} &0 \\
    0 & E_{ij,2}^\top
\end{bmatrix}$. We first notice that for $k = 2,3,4,...$
\begin{equation*}
    (E_{ij,2})^k = (e_i e_{j - n}^\top)^k  = \begin{cases}
        e_ie_i^\top, &\text{ if } i = j - n \\
        0, &\text{ otherwise } 
    \end{cases},
    \qquad (E^\top_{ij,2})^k = (e_{j - n} e_i^\top)^k = \begin{cases}
        e_ie_i^\top, &\text{ if } i = j - n \\
        0, &\text{ otherwise }
    \end{cases}
\end{equation*}
This suggests for $k = 2,3,4,...$,
\begin{align*}
    (E_{ij}\Omega_n)^k &= \begin{cases}
        (-1)^k e_ie_i^\top + e_{n+i}e_{n+i}^\top, &\text{ if } i = j - n\\
        0, &\text{ otherwise }
    \end{cases}  \\ 
    \text{ and } \quad \expm(\theta E_{ij} \Omega_n) &= \begin{cases} 
        I + (e^{-\theta}-1)e_ie_i^\top + (e^\theta-1) e_{n+i}e_{n+i}^\top  &\text{ if } i = j - n \\
        I + \theta E_{ij} \Omega_n, &\text{ otherwise }
    \end{cases} 
\end{align*}
This verifies for $1\leq i \leq n < j \leq 2n$,
\begin{equation*}
    \Retr_X( \theta E_{ij} \Omega_n X) = \begin{cases}
        X + (e^{-\theta}-1) e_i e_i^\top X + (e^\theta - 1) e_{n+i}e_{n+i}^\top X,  &\text{ if } i = j - n\\
        X + \theta E_{ij} \Omega_n X, &\text{ otherwise }
    \end{cases}
\end{equation*}
The proof is now complete.
\end{proof}

\subsection{Doubly stochastic manifold}



For the doubly stochastic manifold, the retraction applies the Sinkhorn algorithm \citep{knight2008sinkhorn} for matrix balancing, i.e., $\Retr_X(t U) = {\rm SK} (X \odot \exp(t U \oslash X))$, where $U$ is a tangent vector belonging to the tangent space $T_X \Pi(\mu, \nu)$ and the Sinkhorn algorithm ${\rm SK}(U)$ iteratively normalize rows and columns of $U$ according to the given marginals \citep{shi2021coupling,cardoso2010exponentials}.

\subsubsection{Proofs}

\begin{proof}[Proof of Proposition \ref{prop_sinkhorn_cd}]
In fact, we can show for any $H_{ijkl} \coloneqq (e_i - e_j)(e_k - e_l)^\top$ for $i \neq j$, $k \neq l$. The coordinate Sinkhorn is a valid retraction along the direction $H_{ijkl}$.
Let $c(t) \coloneqq {\rm cSK}(X \odot \exp( t H_{ijkl} \oslash X))$ and we can immediately see $c(0) = X$. Also, Let $\widetilde{X} \coloneqq X \odot \exp(t H_{ijkl} \oslash X)$. Then $\widetilde{X}$ differs with $X$ in only the entries at $(i,k), (j,k), (i,l), (k,l)$, which forms the $2\times 2$ sub-matrix that we wish to balance. Also, by definition, the marginals are given by $\tilde{\mu} \coloneqq ([X]_{ik}+ [X]_{il}, [X]_{jk} + [X]_{jl})$ and $\tilde{\nu} \coloneqq ([X]_{ik} + [X]_{jk}, [X]_{il} + [X_{jl}])$. It readily holds that $\tilde{\mu}^\top 1_2 = \tilde{\nu}^\top 1_2$. For notational purposes, for any matrix $A \in \sR^{m \times n}$, let $A_{ijkl} \in \sR^{m \times n}$ be the matrix that zeros out the entries except for the $2\times 2$ sub-matrix. Also, we denote $A_{ijkl}^\flat \coloneqq \begin{bmatrix}
    [{A}]_{ik} & [{A}]_{il} \\
    [{A}]_{jk} & [{A}]_{jl}
\end{bmatrix}$ that extracts the corresponding $2\times 2$ sub-matrix.
Then ${\rm cSK}(\widetilde{X})$ reduces to performing Sinkhorn on $\widetilde{X}^\flat_{ijkl}$ with marginals $\tilde{\mu}, \tilde{\nu}$ with other entries of $\widetilde{X}$ unchanged. This is well-defined as the Sinkhorn algorithm converges to the unique doubly stochastic matrix of the form $\diag(u) \widetilde{X}_{ijkl}^\flat \diag(v)$ for some positive vectors $u, v$ \citep{sinkhorn1967diagonal}. This verifies that ${\rm cSK}(\widetilde{X})$ results in a doubly stochastic matrix, which remains on the manifold. Lastly, it remains to show that $c'(0) = H_{ijkl}$. For this, we first have 
\begin{equation*}
    c'(0) = \lim_{t \rightarrow 0} \frac{c(t) - c(0)}{t} = \frac{{\rm cSK}(X \odot \exp(t H_{ijkl} \oslash X)) - X}{t} = \frac{{\rm cSK}(X + t H_{ijkl}) - X}{t},
\end{equation*}
where we use the first-order approximation of the exponential operations. Notice that ${\rm cSK}(X + t H_{ijkl})$ only modifies the $2\times 2$ sub-matrix of $X$ by ${\rm SK}(X_{ijkl}^\flat + t H_{ijkl}^\flat)$. From \citep{douik2019manifold,shi2021coupling}, we have ${\rm SK}(X_{ijkl}^\flat + t H_{ijkl}^\flat) \approx X_{ijkl}^\flat + t H_{ijkl}^\flat$. This suggests 
$\lim_{t \rightarrow 0 } ({{\rm SK}(X^\flat_{ijkl} + t H^\flat_{ijkl}) - X^\flat_{ijkl}})/t = H_{ijkl}^\flat, $
which verifies $c'(0) = H_{ijkl}$.
\end{proof}

\begin{lemma}
\label{lemma_sink_2x2}
Given a positive $2 \times 2$ matrix $A = \begin{bmatrix}
    a & b \\
    c & d
\end{bmatrix}$, the Sinkhorn algorithm on $A$ with marginals $p = [p_1, p_2], q = [q_1, q_2] \in\Delta_2$ converges to $\begin{bmatrix}
    c_{11} a & c_{12} b\\
    c_{21} c & c_{22} d
\end{bmatrix}$ where $c_{12} = p_1/(\kappa a + b), c_{22} = p_2/(\kappa c + d)$, $c_{11} = \kappa c_{12}, c_{21} = \kappa c_{22}$ where $\kappa$ is the positive root of the equation $q_2 ac \kappa^2 + \big( (bc + ad) q_2 - bc p_1 - ad p_2 \big) \kappa - bdq_1 = 0$. 
\end{lemma}

\begin{proof}[Proof of Lemma \ref{lemma_sink_2x2}]
Sinkhorn algorithm converges to the unique doubly stochastic matrix of the form $\diag(u) A \diag(v)$ for $u = [u_1, u_2], v = [v_1, v_2]$. From the constraints, $\diag(u) A \diag(v) 1_2 = p$ and $\diag(v) A^\top \diag(u) 1_2 = q$ we need to solve the quadratic problem 
\begin{equation}\label{eq:sinkhorn_equations}
    \begin{cases}
        (u_1 v_1) a + (u_1 v_2) b = p_1\\
        (u_2 v_1) c + (u_2 v_2) d = p_2\\
        (u_1 v_1) a + (u_2 v_1) c  = q_1 \\
        (u_1 v_2) b + (u_2 v_2) d = q_2.
    \end{cases}
\end{equation}
Let $c_{11} = u_1 v_1, c_{12} = u_1 v_2, c_{21} = u_2 v_1, c_{22} = u_2 v_2$ which transforms (\ref{eq:sinkhorn_equations}) into a set of linear equations for the variables $c_{11}, c_{12}, c_{21}, c_{22}$. The equation system, however, is under-determined and has many solutions. The unique solution that is sought should satisfy $c_{11}/c_{12} = c_{21}/c_{22}  =\kappa$. To this end, from the first two equations, we obtain
\begin{equation*}
    c_{12} = p_1/(\kappa a + b), \quad c_{22} = p_2/(\kappa c + d).
\end{equation*}
Substituting the expressions to the last equation yields $\frac{b p_1}{\kappa a + b} + \frac{d p _2}{\kappa c + d} = q_2$, which we solve for $\kappa$ as the positive root of $q_2 ac \kappa^2 + \big( (bc + ad) q_2 - bc p_1 - ad p_2 \big) \kappa - bdq_1 = 0$. 
\end{proof}

\section{Formal developments and proofs for Section \ref{algo_sect}} \label{appendix:sec:proofs}

\subsection{Developments}
\begin{assumption}
\label{main_conv_assump}
Consider a neighbourhood $\gX \subseteq \M$ that contains a critical point. 
\begin{enumerate}[label=\ref{main_conv_assump}.\arabic*, labelindent=10pt, noitemsep]
    \item The basis and its projection are bounded. Let the projection onto the basis $B_{\ell,X}$ be $\gP_{B_{\ell,X}}(U) \coloneqq \langle U, B_{\ell, X} \rangle_X B_{\ell, X}$. There exists constant $c_b , c_p> 0$ such that $\forall X \in \gX, U \in T_X\M, \ell \in \gI$, $\|B_{\ell, X}\|^2_X \leq c_b$ and $\sum_{\ell \in \gI} \|\gP_{B_{\ell, X}} U \|^2_X \geq c_p \| U\|^2_X$.  \label{assump_bound_basis}
    \item The objective $f$ is retraction $L$-smooth in $\gX$, i.e., $f(\Retr_X(U)) - f(U) - \langle \grad f(X), U \rangle_X \leq \frac{L}{2} \| U\|_X^2$, $\forall X \in \gX$ and $U \in T_X\M$ such that $\Retr_X(U) \in \gX$. \label{assump_retr_L_smooth}
\end{enumerate}
\end{assumption}

\begin{remark}
Assumption \ref{assump_bound_basis} requires the basis has a bounded norm and the projection of any tangent vector onto the basis does not vanish. Such an assumption is manifold-specific and we can verify that $\| B_{\ell, X} \|_X^2$ has an upper bound (e.g., for Stiefel and Grassmann, $\| B_{\ell, X} \|_X^2 \leq \| H_{i,j} \|^2_F = {2}$). Then we note that the second requirement trivially holds for orthonormal basis due to the decomposition of $U$ and Jensen's inequality. For non-orthonormal basis, this assumption also holds as long as projection of a tangent vector does not vanish. Assumption \ref{assump_retr_L_smooth} can also be satisfied by the compactness of the domain, e.g., we can take $L$ in Assumption \ref{assump_retr_L_smooth} to be $L = \max_{X \in \mathcal X, U \in T_X\mathcal{M}: \| U\|_X = 1} \frac{d^2 f({\rm Retr}_X(t U))}{dt^2}$.  These are all reasonable assumptions within a compact neighbourhood $\gX$. 
\end{remark}

\noindent\textbf{Theorem \ref{thm_rcd_random}} (Formal)\textbf{.}
\textit{Under Assumption \ref{main_conv_assump}, consider RCD algorithm with $S = 1$ and $\ell_k^s$ selected uniformly at random from $\gI$. Then suppose $\eta = \frac{1}{Lc_b}$, it satisfies $\min_{ 0\leq k \leq K-1} \sE \| \grad f(X_k) \|^2_{X_k} \leq \frac{2L c_b^2 c_p^{-1} |\gI| \Delta_0}{K}$, where $\Delta_0 \coloneqq f(X_0) - f^*$.}

To analyze the cyclic variant, we further require the assumption that bounds the difference between Riemannian distance and distance induced by general retraction. In addition, we require the gradient Lipschitzness.

\begin{assumption}
\label{assump_cyclic_rcd}
Under the same settings as in Assumption \ref{main_conv_assump},
\begin{enumerate}[label=\ref{assump_cyclic_rcd}.\arabic*, labelindent=20pt,noitemsep]
\item For all $X, Y = \Retr_X(U) \in \gX$, there exists constants $\vartheta_0, \vartheta_1 > 0$ such that $ \vartheta_0\| U \|^2_X \leq \dist^2(X,Y) \leq \vartheta_1 \| U \|^2_X$.  \label{assump_dist_retr}

\item The objective has retraction $L_g$-Lipschitz gradient, i.e., $\| \grad  f(X) - \gT_Y^X \grad f(Y) \|^2_X \leq L_g \| U \|^2_X$, $\forall X, Y = \Retr_X(U) \in \gX$ and $\gT_X^Y$ is the an isometric vector transport that satisfies $\langle \gT_X^Y U, \gT_X^Y V \rangle_Y = \langle U, V\rangle_X$, $\forall X, Y \in \gX, U, V \in T_X\M$. \label{grad_lips_assump}

\item For any fixed coordinate index $\ell \in \gI$, there exists a constant $\upsilon > 0$ such that for all $X, Y \in \gX$, $V \in T_Y\M$, $\| \gP_{B_{\ell, X}} \gT_{Y}^X V \|^2_X \geq \upsilon \| \gP_{B_{\ell, Y}} V \|^2_Y$. \label{assump_same_basis_pro}
\end{enumerate}
\end{assumption}

\begin{remark}
Assumption \ref{assump_dist_retr} bounds the difference Riemannian distance (relating to inverse exponential map) and the inverse retraction. Because retraction is a first-order approximation to the exponential map, this assumption naturally holds when the domain is sufficiently small (see \cite{huang2015riemannian,sato2019riemannian}). Assumption \ref{grad_lips_assump} is further required because when general retraction is used, gradient Lipschitzness is not equivalent to function smoothness. Assumption \ref{assump_same_basis_pro} further claims that the difference between the same coordinate basis on different tangent spaces is bounded. We note that the RHS is identical to $\| \mathcal P_{\mathcal T_Y^X B_{\ell,Y}} \mathcal T_{Y}^X V \|$  due to the isometric vector transport. Then it reduces to whether $\mathcal T_Y^X B_{\ell,Y}$ and $B_{\ell, X}$ are related, which is expected because due to the compactness of the domain, $X$ is bounded from $Y$. This allows to establish the convergence for cyclic selection of basis. 
\end{remark}

\noindent\textbf{Theorem \ref{thm_cyclic_rcd}} (Formal)\textbf{.}
\textit{Under Assumption \ref{main_conv_assump} and consider RCD algorithm with $S = |\gI|$ and $\ell_k^s =s+1$ for $s = 0, ..., |\gI|-1$. Then selecting $\eta = \frac{1}{L c_b}$ gives $\min_{0 \leq  k \leq K-1} \| \grad f(X_k) \|_{X_k} \leq \frac{C \Delta_0}{K}$, where $C = 4 L c_b^2 c_p^{-1} \upsilon^{-1} (1 + |\gI|^2 c_b^{-1} L^{-2} L_g  \vartheta_1 \vartheta_0^{-1})$.} \\

\noindent\textbf{Theorem \ref{thm_rand_cyc_rcdlin}} (Formal)\textbf{.}
\textit{Under Assumption \ref{main_conv_assump}, \ref{assump_cyclic_rcd} and further let $\omega_0, \omega_1 > 0$, such that for any fixed epoch $k$, {$\omega_0 \langle \nabla f(X_k^s), B_{\ell_k^s} \rangle^2 \leq  \theta_k^s \langle \nabla f(X_k^s), B_{\ell_k^s} \rangle \leq \omega_1 \langle \nabla f(X_k^s), B_{\ell_k^s} \rangle^2$,} $\forall s \leq S_{\rm max}-1$. 
}

\textit{(Randomized). Consider RCDlin algorithm with $1< S \leq S_{\rm max}$ and $\ell_k^s$ selected uniformly at random from $\gI$. Then choosing $\eta = \frac{\omega_0}{L c_b \omega_1^2}$, we obtain $\min_{0 \leq k \leq K-1, 0\leq s \leq S-1} \sE \| \grad f(X_k^s) \|_{X_k^s}^2 \leq \frac{2L c_b^2 c_p^{-1} w_0^{-2} w_1^2 |\gI| \Delta_0}{KS}$. 
}

\textit{(Cyclic). Suppose $S_{\rm max} \geq |\gI|$ and consider RCDlin algorithm with $S = |\gI|$ and $\ell_k^s = s+1$ for $s = 0,..., |\gI|-1$. Then choosing $\eta = \frac{\omega_0}{L c_b \omega_1^2}$, we have $\min_{0\leq k \leq K-1} \| \grad f(X_k^s) \|^2_{X_k^s} \leq \frac{\widetilde{C}\Delta_0}{K}$, where $\widetilde{C} = 4L c_b^2 \omega_1^2 \omega_0^{-2} c_p^{-1} \nu^{-1} (1 + |\gI|^2 c_b^{-1} L^{-2}L_g \vartheta_1 \vartheta_0^{-1})$.
}

\begin{remark}
{We finally remark that the proof ideas of cyclic and randomized RCD follow from classic developments of coordinate descent \cite{wright2015coordinate} in the Euclidean space by showing sufficient descent in the objective function. On general manifolds, in order to generalize the proof ideas, we further require the assumptions outlined in Assumption \ref{main_conv_assump}, \ref{assump_cyclic_rcd}. In particular, for cyclic selection rule, we require Assumption \ref{assump_same_basis_pro} to relate bases from different tangent spaces. Similar assumptions have been considered in \cite{gutman2023coordinate} for showing convergence of deterministic subspace descent algorithms on manifolds (see ($C,r$)-norm condition).}
\end{remark}

\subsection{Proofs}

\begin{proof}[Proof of Theorem \ref{thm_rcd_random}]
Because $S=1$ and by retraction $L$-smoothness, we have 
\begin{align*}
    f(X_{k+1}) - f(X_k) &\leq  - \eta \langle \grad f(X_k), \theta_k B_{\ell_k} \rangle_{X_k} + \frac{\eta^2 \theta_k^2 L}{2} \| B_{\ell_k} \|^2_{X_k} \\
    &\leq - \eta \theta_k^2 \frac{\| B_{\ell_k} \|_{\ell_k}^2}{c_b} + \frac{\eta^2\theta_k^2L}{2} \| B_{\ell_k} \|^2_{X_k} \\
    &=  - \frac{\| \gP_{B_{\ell_k}}\grad f(X_k) \|_{X_k}^2}{2Lc_b^2},
\end{align*}
where we use the assumption $\|B_{\ell_k} \|^2_{X_k} \leq c_b$, and choose $\eta = \frac{1}{Lc_b}$. Taking expectation with respect to $\ell_k$, we have 
\begin{align*}
    \sE_{\ell_k}[f(X_{k+1})] - f(X_k) &\leq - \frac{1}{2Lc_b^2} \sE_{\ell_k}[ \| \gP_{B_{\ell_k}} \grad f(X_k) \|_{X_k}^2 ] \\
    &= -\frac{1}{2|\gI|Lc_b^2} \sum_{\ell \in \gI} {\| \gP_{B_{\ell, X_k}} \grad f(X_k) \|^2_{X_k}} \\
    &\leq -\frac{c_p}{2 |\gI| L c_b^2}  \| \grad f(X_k) \|^2_{X_k}. 
\end{align*}
Telescoping this inequality and taking full expectation yields
\begin{equation*}
    \frac{1}{K} \sum_{k=0}^{K-1} \sE \| \grad f(X_k) \|_{X_k}^2 \leq \frac{2L c_b^2 c_p^{-1} |\gI| \Delta_0}{K},
\end{equation*}
where we let $\Delta_0 \coloneqq f(X_0) - f^*$. 
\end{proof}

\begin{lemma}
\label{proj_bound_lem}
Under Assumption \ref{assump_bound_basis}, we have $\| \gP_{B_{\ell, X}} U \|_X \leq c_b \| U\|_X$, $\forall X \in \gX, \ell \in \gI$.
\end{lemma}
\begin{proof}[Proof of Lemma \ref{proj_bound_lem}]
$\| \gP_{B_{\ell, X}} U \|^2_X \leq c_b \langle U, B_\ell(X) \rangle_X^2 = c_b \big\langle \langle U, B_\ell(X)  \rangle_X B_\ell(X) , U \big\rangle_X \leq c_b \|\gP_{B_{\ell, X}} U \|_X \| U\|_X$. Cancelling $\| \gP_{B_{\ell, X}} U \|_X$ on both sides completes the proof. 
\end{proof}

\begin{proof}[Proof of Theorem \ref{thm_cyclic_rcd}]
We first focus on a single epoch $k$ and for notation simplicity, we let $\gT_{X_k^s}^{X_k} = \gT_{s \rightarrow 0}$ and $\gT_{X_k}^{X_k^s} = \gT_{0 \rightarrow s}$. Similarly from retraction $L$-smoothness,
\begin{equation*}
    f(X_{k}^{s+1}) - f(X_k^{s}) \leq - \frac{1}{2Lc_b^2} \| \gP_{B_{\ell_k^s}} \grad f(X_k^s) \|_{X_k^s}^2
\end{equation*}
with stepsize $\eta = \frac{1}{Lc_b}$. Summing over $s = 0, ..., S-1$, we can bound 
\begin{equation*}
    f(X_{k+1}) - f(X_k) \leq - \frac{1}{2Lc_b^2} \sum_{s=0}^{S-1} \| \gP_{B_{\ell_k^s}} \grad f(X_k^s) \|_{X_k^s}^2.
\end{equation*}

Then it remains to bound the RHS. From $L_g$-Lipschitz, 
\begin{align*}
    \|   \grad f(X_k^s) -  \gT_{0 \rightarrow s}\grad f(X_k)   \|_{X_k^s}^2 \leq L_g \| \Retr_{X_k^s}^{-1}(X_k) \|_{X_k^s}^2 &\leq L_g \vartheta_0^{-1} \dist^2(X_k, X_k^s) \\
    &\leq L_g \vartheta_0^{-1} \big( \sum_{j = 1}^s \dist(X_k^{j-1}, X_k^j) \big)^2\\
    &\leq \eta^2 L_g \vartheta_1 \vartheta_0^{-1} s \sum_{j=0}^{s-1} \| \gP_{B_{\ell_k^j}}(\grad f(X_k^j)) \|^2_{X_k^j}.
\end{align*}
where we use Assumption \ref{assump_cyclic_rcd} and triangle inequality of Riemannian distance. 

Now we can show 
\begin{align*}
    \|  \gP_{B_{\ell_k^s}} \gT_{0 \rightarrow s} \grad f(X_k) \|_{X_k^s}^2 &\leq 2\| \gP_{B_{\ell_k^s}} \gT_{0 \rightarrow s} \grad f(X_k) - \gP_{B_{\ell_k^s}} \grad f(X_k^s)  \|^2_{X_k^s} + 2\| \gP_{B_{\ell_k^s}} \grad f(X_k^s) \|^2_{X_k^s} \\
    &\leq 2 c_b \| \gT_{0 \rightarrow s} \grad f(X_k) - \grad f(X_k^s) \|^2_{X_k^s} + 2 \| \gP_{B_{\ell_k^s}} \grad f(X_k^s) \|^2_{X_k^s} \\
    &\leq 2 c_b \eta^2 L_g \vartheta_1 \vartheta_0^{-1} S \sum_{j=0}^{s-1} \| \gP_{B_{\ell_k^j}}(\grad f(X_k^j)) \|^2_{X_k^j} + 2 \| \gP_{B_{\ell_k^s}} \grad f(X_k^s) \|^2_{X_k^s} 
\end{align*}
where the second inequality is due to Lemma \ref{proj_bound_lem}. Summing this inequality from $s = 0,..., S-1$ gives
\begin{align*}
    &\sum_{s = 0}^{S-1} \| \gP_{B_{\ell_k^s}} \gT_{0\rightarrow s} \grad f(X_k) \|_{X_k^s}^2 \\
    &\leq 2 \sum_{s = 0}^{S-1} \| \gP_{B_{\ell_k^s}} \grad f(X_k^s) \|^2_{X_k^s} + 2 c_b \eta^2 L_g \vartheta_1 \vartheta_0^{-1} |\gI|\sum_{s = 0}^{S-1} \sum_{j=0}^{s-1} \| \gP_{B_{\ell_k^j}}(\grad f(X_k^j)) \|_{X_k^j} \\
    &\leq 2(1 + |\gI|^2 c_b \eta^2 L_g \vartheta_1 \vartheta_0^{-1}) \sum_{s=0}^{S-1} \| \gP_{B_{\ell_k^s}} \grad f(X_k^s) \|^2_{X_k^s} 
\end{align*}
where we notice $S = |\gI|$. Also due to the cyclic selection of $\ell_{k}^s$, we can see the LHS is 
$$ \sum_{s = 0}^{S-1} \| \gP_{B_{\ell_k^s}} \gT_{0\rightarrow s} \grad f(X_k) \|_{X_k^s}^2 \geq \upsilon \sum_{\ell \in \gI}  \| \gP_{B_{\ell, X_k}} \grad f(X_k) \|_{X_k}^2 \geq c_p \upsilon \| \grad f(X_k) \|^2_{X_k},$$
where we use Assumption \ref{assump_bound_basis} and \ref{assump_same_basis_pro}. Combining with previous results, we finally obtain 
\begin{equation*}
    f(X_{k+1}) - f(X_k) \leq - \frac{1}{2Lc_b^2} \sum_{s=0}^{S-1} \| \gP_{B_{\ell_k^s}} \grad f(X_k^s) \|_{X_k^s}^2 \leq - \frac{1}{C} \| \grad f(X_k) \|_{X_k}^s,
\end{equation*}
where $C \coloneqq 4 L c_b^2 c_p^{-1} \upsilon^{-1} (1 + |\gI|^2 c_b^{-1} L^{-2} L_g  \vartheta_1 \vartheta_0^{-1})$. Telescoping this inequality completes the proof.
\end{proof}

\begin{proof}[Proof of Theorem \ref{thm_rand_cyc_rcdlin}]
 For the randomized setting, by retraction $L$-smoothness,
 \begin{align*}
     f(X_k^{s+1}) - f(X_k^s) &\leq - \eta \langle \grad f(X_k^s), \theta_k^s B_{\ell_k^s} \rangle_{X_k^s} + \frac{\eta^2 {\theta_k^s}^2 L}{2} \| B_{\ell_k^s} \|^2_{X_k^s} \nonumber\\
     &= - \eta \theta_k^s \langle \nabla f(X_k^s), B_{\ell_k^s} \rangle + \frac{\eta^2 {\theta_k^s}^2 L}{2} \| B_{\ell_k^s} \|^2_{X_k^s} \nonumber\\
     &\leq - \eta {\omega_0} \langle  \nabla f(X_k), B_{\ell_k^s}\rangle^2 + \frac{\eta^2 \omega_1^2 L}{2} \| \langle \nabla f(X_k^s), B_{\ell_k^s} \rangle B_{\ell_k^s} \|^2_{X_k^s} \nonumber \\
     &\leq \Big( - \frac{\eta \omega_0}{c_b} + \frac{\eta^2 \omega_1^2 L}{2} \Big) \| \langle \nabla f(X_k^s), B_{\ell_k^s} \rangle B_{\ell_k^s} \|^2_{X_k^s} \nonumber\\
     &= - \frac{\omega_0^2}{2 L c_b^2 \omega_1^2} \| \gP_{B_{\ell_k^s}} \grad f(X_k^s) \|_{X_k^s}^2
 \end{align*}
 where we choose $\eta = \frac{\omega_0}{L c_b \omega_1^2}$. The second inequality follows from the assumption $\omega_0 \langle \nabla f(X_k^s), B_{\ell_k^s} \rangle^2 \leq  \theta_k^s \langle \nabla f(X_k^s), B_{\ell_k^s} \rangle \leq \omega_1 \langle \nabla f(X_k^s), B_{\ell_k^s} \rangle^2$ and the third inequality is due to Assumption \ref{assump_bound_basis}. Following the similar proof strategy, we obtain the desired result. For the cyclic setting, the bound also readily follows by using the above result.
\end{proof}

\end{document}